\documentclass[11pt, oneside]{article}   	
\usepackage{geometry}                		
\usepackage{graphicx}				
								
\usepackage[all]{xy}
\CompileMatrices

\usepackage{amssymb}
\usepackage{amsmath}
\usepackage{stmaryrd}
\usepackage{amsthm}
\usepackage{tikz-cd}
\usepackage{extarrows}
\usepackage[mathscr]{euscript}
\usepackage{mathrsfs} 
\usepackage{verbatim}
\usepackage[OT2,T1]{fontenc}
\DeclareSymbolFont{cyrletters}{OT2}{wncyr}{m}{n}
\DeclareMathSymbol{\Sha}{\mathalpha}{cyrletters}{"58}
\DeclareMathSymbol{\Che}{\mathalpha}{cyrletters}{"51}

\usepackage{calligra}
\usepackage{mathrsfs}
\newcommand{\calHom}{\mathscr{H}\mathit{om}}
\newcommand{\calMor}{\mathscr{M}\mathit{or}}

\newcommand{\Ga}{{\mathbf{G}}_{\rm{a}}}
\newcommand{\Gm}{{\mathbf{G}}_{\rm{m}}}
\DeclareMathOperator{\red}{red}

\DeclareMathOperator{\imp}{Imp}
\DeclareMathOperator{\Gal}{Gal}

\DeclareMathOperator{\Pic}{Pic}
\DeclareMathOperator{\Jac}{Jac}

\DeclareMathOperator{\Ext}{Ext}
\DeclareMathOperator{\Split}{split}
\DeclareMathOperator{\Hom}{Hom}
\DeclareMathOperator{\End}{End}

\DeclareMathOperator{\R}{R}
\DeclareMathOperator{\sm}{sm}

\DeclareMathOperator{\im}{im}

\newcommand*{\Z}{\ensuremath{\mathbf{Z}}}                        
\newcommand*{\F}{\ensuremath{\mathbf{F}}}                        
\newcommand*{\A}{\ensuremath{\mathbf{A}}}                        
\renewcommand*{\P}{\ensuremath{\mathbf{P}}}                        
\newcommand*{\calO}{\mathcal{O}}                                  
\newcommand*{\address}{Einstein Institute of Mathematics, The Hebrew University of Jerusalem, Edmond J. Safra Campus, 91904, Jerusalem, Israel}
\newcommand*{\email}{zevrosengarten@gmail.com}

\usepackage{rotating}

\usepackage{bm}

\numberwithin{equation}{section}

\newtheorem{theorem}{Theorem}[section]

\newtheorem{lemma}[theorem]{Lemma}
\newtheorem{proposition}[theorem]{Proposition}
\newtheorem{corollary}[theorem]{Corollary}

\theoremstyle{definition}
  \newtheorem{definition}[theorem]{Definition}

\theoremstyle{remark}
  \newtheorem{remark}[theorem]{Remark}

\usepackage[OT2,T1]{fontenc}

\tikzset{commutative diagrams/.cd,
mysymbol/.style={start anchor=center,end anchor=center,draw=none}
}

\usepackage{stmaryrd}
\usepackage{hyperref}

\title{\textbf{RIGIDITY AND UNIRATIONAL GROUPS}}
\author{Zev Rosengarten \thanks{While completing this work, the author was supported by a Zuckerman Postdoctoral Scholarship. \newline
While completing this work, the author was supported by Israel Science Foundation Grant No.\,2083/24.
}}

\date{}
\begin{document}
\maketitle

\begin{abstract}
We prove a rigidity theorem for morphisms from products of open subschemes of the projective line into solvable groups  not containing a copy of $\Ga$ (for example, wound unipotent groups). As a consequence, we deduce several structural results about unirational group schemes, including that unirationality for group schemes descends through separable extensions. We also apply the main result to prove that permawound unipotent groups are unirational and -- when wound -- commutative.
\end{abstract}

\setcounter{tocdepth}{1}
\tableofcontents{}

\section{Introduction}

One of the foundational results in the theory of abelian varieties is the following ``rigidity lemma.''

\begin{lemma}[{\cite[Ch.\,II, \S4, Rigidity Lemma]{mumford}}]
\label{rigidityproper}
Let $K$ be an algebraically closed field, $X$, $Y$, and $Z$ integral, separated $K$-schemes of finite type with $X$ proper. Let $f \colon X \times Y \rightarrow Z$ be a $K$-morphism such that, for some $y_0 \in Y$, $f(X \times \{y_0\})$ is a single point of $Z$. Then there is a morphism $g \colon Y \rightarrow Z$ such that $f = g \circ \pi_2$, where $\pi_2\colon X \times Y \rightarrow Y$ is the projection.
\end{lemma}

(For a slightly more general version, see \cite[Th.\,1.7.1]{conradcoursenotes}.) This rigidity theorem has many important consequences in the theory of abelian varieties, such as the fact that any $K$-scheme morphism between abelian varieties which preserves identities is a $K$-group scheme homomorphism, commutativity of abelian varieties, etc. The main result of the present paper is a somewhat analogous rigidity result for maps from products of open subschemes of the projective line into solvable groups not containing a copy of $\Ga$ (such as wound unipotent groups). As with the rigidity lemma \ref{rigidityproper}, we will show that this result has several interesting applications -- in this case, to the study of unirational groups.

Before we state the main result, we require a definition. For a finite extension of fields $L/K$ of characteristic $p > 0$, we define the {\em degree of imperfection} of $L/K$, denoted $\imp(L/K)$, to be the nonnegative integer $r$ defined by the equality $[L: KL^p] = p^r$, where $KL^p$ denotes the compositum inside $L$ of the subfields $K$ and $L^p$. (The degree in question is a power of $p$ because the extension $L/KL^p$ is purely inseparable.) As we shall see (Proposition \ref{degofprimminnumofgens}), the degree of imperfection is usually the minimum number of elements required to generate $L$ over $K$. This is always true if $L/K$ is purely inseparable \cite[Th.\,6]{beckermaclane}. An important point for us is that there is a natural way of extending this definition from finite field extensions $L$ of a given field $K$ of characteristic $p$ to finite reduced $K$-algebras. Indeed, if $A$ is a finite reduced $K$-algebra, then we have a $K$-algebra isomorphism $A \simeq \prod_{i=1}^n L_i$ for some finite field extensions $L_i/K$. Choose a tuple $(F, \sigma_1, \dots, \sigma_n)$ consisting of a field extension $F/K$ and $K$-embeddings $\sigma_i\colon L_i \hookrightarrow F$. Then we may form the compositum $L := \sigma_1(L_1)\dots\sigma_n(L_n)$ inside $F$ of the fields $\sigma_i(L_i)$, and we then define the degree of imperfection $\imp(A/K)$ to be $\imp(L/K)$. Of course, we must show that this definition is independent of all choices (see Proposition \ref{degofprimindofembeddings}). This is relatively straightforward if $K$ is separably closed, but less clear in general. The main result of the present paper is the following rigidity statement.

\begin{theorem}[Rigidity Theorem]
\label{rigiditytheorem}
Let $K$ be a field of characteristic $p > 0$, let $\overline{X}_1, \dots, \overline{X}_n$ be smooth proper geometrically connected curves over $K$, and let $X_i \subset \overline{X}_i$ be dense open subschemes for $1 \leq i \leq n$ with closed complement $D_i := \overline{X}_i \backslash X_i$. Let $D := D_1 \sqcup \dots \sqcup D_n$, and $r := \imp(\Gamma(D, \calO_D)_{\red}/K)$, the degree of imperfection of $D/K$. For each $1 \leq i \leq n$, let $x_i \in X_i(K)$. Finally, let $G$ be a solvable $K$-group scheme of finite type not containing a $K$-subgroup scheme $K$-isomorphic to $\Ga$. If $n > r$, then the only $K$-morphism $f\colon X_1 \times \dots \times X_n \rightarrow G$ such that, for each $1 \leq i \leq n$, $f|X_1 \times \dots X_{i-1} \times \{x_i\} \times X_{i+1} \times \dots \times X_n = 1_G$ is the constant map to the identity $1_G$.
\end{theorem}

As we shall see, the key case is when all of the $\overline{X}_i$ are $\P^1$, so that the image of the map is unirational. (Recall that a finite type $K$-scheme $X$ is said to be {\em unirational} if there is a dominant (that is, not factoring through any smaller closed subscheme) rational map $\P^n \dashrightarrow X$ for some integer $n > 0$. In particular, $X$ is geometrically integral.)

The rigidity theorem \ref{rigiditytheorem} has many applications. We first use it to prove another rigidity theorem over fields of finite degree of imperfection. We will see (Proposition \ref{degimpleqdegimperf}) that any finite reduced algebra $A$ over a field of degree of imperfection $r$ satisfies $\imp(A/K) \leq r$. Combining this with the rigidity theorem \ref{rigiditytheorem}, we will prove the following result when $K$ has finite degree of imperfection.

\begin{theorem}
\label{rigidityfindegofimp}
Let $K$ be a field of degree of imperfection $r$, and let $X_1, \dots, X_n$ be unirational $K$-schemes. For each $X_i$, let $x_i \in X_i^{\sm}(k)$ be a $k$-point in the smooth locus of $X_i$. Finally, let $G$ be a solvable affine $K$-group scheme of finite type not containing a $K$-subgroup scheme $K$-isomorphic to $\Ga$. If $n > r$, then the only $K$-morphism $f\colon X_1 \times \dots \times X_n \rightarrow G$ such that, for each $1 \leq i \leq n$, $f|X_1 \times \dots X_{i-1} \times \{x_i\} \times X_{i+1} \times \dots \times X_n = 1_G$ is the constant map to the identity $1_G$.
\end{theorem}

Consider the case $r = 1$ of the above theorem, and suppose given a $K$-scheme morphism $f \colon H \rightarrow G$ between $K$-group schemes such that $f(1_G) = 1_H$, with $G$ as in the theorem and $H$ unirational. Applying Theorem \ref{rigidityfindegofimp} to the map $g\colon H \times H \rightarrow G$ given by the formula $g(h_1, h_2) := f(h_1h_2)f(h_2)^{-1}f(h_1)^{-1}$ shows that $f$ is a homomorphism. That is, we obtain the following theorem.

\begin{theorem}
\label{homdegofimp1}
Let $K$ be a field of degree of imperfection $1$. If $G, H$ are finite type $K$-group schemes with $H$ unirational and $G$ solvable and not containing a $K$-subgroup scheme $K$-isomorphic to $\Ga$, then any $K$-scheme morphism $f\colon H \rightarrow G$ such that $f(1_H) = 1_G$ is a homomorphism.
\end{theorem}

Here is another application. For a smooth $K$-group $G$, consider the descending central series $\mathscr{D}_nG \trianglelefteq G$ defined inductively by the formulas $\mathscr{D}_0G := G$ and $\mathscr{D}_{n+1}G := [G, \mathscr{D}_nG]$ for any $n \geq 0$. A group is called nilpotent if $\mathscr{D}_n G = 1$ for sufficiently large $n$. Any smooth unipotent group $U$ is nilpotent, but one cannot in general bound the nilpotency class of $U$ (the minimum integer $n \geq 0$ such that $\mathscr{D}_nG = 1$). For example, the group of upper triangular $n\times n$ unipotent matrices has nilpotency class $n - 1$. On the other hand, using the rigidity theorem \ref{rigidityfindegofimp} over fields of finite degree of imperfection, we may show that, over fields of degree of imperfection $r$, every solvable unirational group not containing $\Ga$ has nilpotency class at most $r$:

\begin{theorem}
\label{nilpotclass}
Let $K$ be a field of degree of imperfection $r$, and let $G$ be a solvable unirational $K$-group scheme not containing a $K$-subgroup scheme $K$-isomorphic to $\Ga$. Then $\mathscr{D}_rG = 1$.
\end{theorem}

\begin{proof}
An easy induction using the identity
\[
[g, xy] = [g, x][xgx^{-1}, xyx^{-1}]
\]
shows that $\mathscr{D}_rG$ is generated (as a $K$-group) by the $(r+1)$-fold commutator map $G^{r+1} \rightarrow G$ defined by $$(g_1, \dots, g_{r+1}) \mapsto [g_1, [g_2, \dots, g_{r+1}]]\dots].$$ This map has the property that if any of the components is restricted to $1 \in G$, then the map becomes the constant map to $1 \in G$. Theorem \ref{rigidityfindegofimp} then completes the proof.
\end{proof}

The next application will be to show that unirationality descends through separable extensions for algebraic groups. Achet, using geometric class field theory, showed that, for a {\em commutative} $K$-group scheme $G$ of finite type, if $L/K$ is a separable field extension such that $G$ becomes unirational over $L$, then in fact $G$ is already unirational over $K$ \cite[Th.\,2.3]{achet}. Scavia, following Colliot-Th\'el\`ene, gave a more elementary proof of the same fact \cite[Lem.\,2.1]{scavia}. We extend this result to arbitrary finite type $K$-groups (no commutativity hypotheses).

\begin{theorem}
\label{unirationalitydescends}
Let $L/K$ be a $($not necessarily algebraic$)$ separable extension of fields, and let $G$ be a finite type $K$-group scheme such that $G_L$ is unirational over $L$. Then $G$ is unirational over $K$.
\end{theorem}

Using Theorem \ref{unirationalitydescends}, it is not hard to deduce that the formation of the maximal unirational subgroup scheme of a finite type group scheme $G$ commutes with separable extension (Corollary \ref{maxunirinvsepext}), and that it is a normal subgroup if $G$ is smooth (Corollary \ref{maxunirrnomal}).

We also prove the following theorem, which says something nontrivial about the structure of unirational wound groups.

\begin{theorem}
\label{genbycommunirsubgps}
If $U$ is a unirational wound unipotent $K$-group scheme, then $U$ is generated by its commutative unirational $K$-subgroups.
\end{theorem}

The unirationality assumption in Theorem \ref{genbycommunirsubgps} is crucial in order to conclude that $G$ is generated by smooth connected commutative $K$-subgroups. Indeed, in \cite[Example B.2.9]{cgp}, Conrad, Gabber, and Prasad give an example of a two-dimensional, non-commutative, wound unipotent group $U$ over every imperfect field $K$. We claim that for any two-dimensional, smooth, connected, non-commutative unipotent $K$-group $U$, the only smooth connected $K$-subgroups of $U$ are $1$, $\mathscr{D}U$, and $U$. In particular, $U$ is not generated by its smooth connected commutative $K$-subgroups. To see the claim, first note that -- because $U$ is nilpotent and ${\rm{dim}}(\mathscr{D}U) = 1$ -- $\mathscr{D}U$ is central in $U$. Now let $1 \subsetneq V \subsetneq U$ be a smooth connected $K$-subgroup, so $V$ is one-dimensional, hence commutative. We must show that $V = \mathscr{D}U$. If the map $V \rightarrow U^{{\rm{ab}}} = U/\mathscr{D}U$ is nonzero, then it is surjective, because ${\rm{dim}}(U^{\rm{ab}}) = 1$.  It would then follow that $V$ and $\mathscr{D}U$ generate $U$. Since $V$ is commutative and $\mathscr{D}U \subset U$ is central, it would follow that $U$ is commutative, contrary to hypothesis. Therefore, $V \subset \mathscr{D}U$, hence $V = \mathscr{D}U$ because both groups are one-dimensional.

Finally, we apply the rigidity theorem \ref{rigidity} to the study of permawound unipotent groups. These groups were introduced in \cite{rosasuniv} as a useful tool in the study of problems pertaining to unipotent groups. Let us recall their definition.

\begin{definition}$($\cite[Def.\,1.2]{rosasuniv}$)$
\label{asunivdef}
We say that a smooth unipotent group scheme $U$ over a field $K$ is {\em permawound} when the following condition holds: For every right-exact sequence of finite type $K$-group schemes
\[
U \longrightarrow E \xlongrightarrow{\pi} \Ga \longrightarrow 1,
\]
$E$ contains a $K$-subgroup scheme $K$-isomorphic to $\Ga$.
\end{definition}

When $U$ is wound -- which is the main case of interest -- the above definition can be rephrased as follows: For every right-exact sequence as in the definition, there is a surjective homomorphism $f\colon \Ga \rightarrow \Ga$ such that $f$ factors as $\pi\circ g$ for some $g\colon\Ga \rightarrow E$.

Some of the basic properties of permawound groups will be recalled in \S\ref{asunivsection}. For now, suffice it to say that they form an important tool in the study of wound unipotent groups. As an illustration of this, they will play a crucial role in the proof of the rigidity theorem \ref{rigiditytheorem}. See also \cite{rosasuniv}, where several applications are given.

The main result that we prove about permawound groups in this paper is the following.

\begin{theorem}
\label{asunivunirrcomm}
Let $K$ be an imperfect field.
\begin{itemize}
\item[(i)] Every permawound unipotent $K$-group is unirational.
\item[(ii)] Every wound permawound unipotent $K$-group is commutative.
\end{itemize}
\end{theorem}

This paper is organized as follows. In \S\ref{degimpsection}, we introduce the notion of degree of imperfection of a finite extension and investigate its properties. Subsequently, \S\S\ref{asunivsection}--\ref{deductionsection} are concerned with the proof of the main rigidity theorem \ref{rigiditytheorem}. At a more granular level, \S\ref{asunivsection} explains how the properties of permawound groups allow one to reduce the proof of the rigidity theorem to the case of maps into a particular group $G$, while \S\ref{modcompsection} carries out a computation with the moduli space of mappings from certain rational curves into this group $G$. This section is quite technical, so the reader may wish to skip it on a first reading, taking Proposition \ref{modspacescomp} as a black box. This computation plays a fundamental role in the proof of the rigidity theorem when the rational curves appearing in the theorem statement are of a particular form. This proof is carried out in \S\ref{rigalmostcompleteratlsection}. The proof of the rigidity theorem in the general case is then obtained from this case in \S\ref{deductionsection} by means of geometric class field theory. In \S\ref{appunirgpssection}, we apply the rigidity theorem to the study of unirational groups. Many of our main results are proven there, and the reader may wish to skip to this section on a first reading to see what rigidity is good for. The remainder of the paper is devoted to the study of permawound unipotent groups. The next section, \S\ref{torsorssection}, is somewhat technical. Its main result is a proposition that allows one to lift maps from rational curves through torsors for certain group schemes. This is important for lifting unirationality from quotient groups. Then \S\ref{Vnlambdasection} studies certain groups which are important because they surject onto all commutative, $p$-torsion permawound groups. Finally, we bring everything together in \S\ref{appstoasunivsec} to prove our main structure theorems for permawound groups.

\subsection{Acknowledgements}

I am happy to thank the anonymous referee for his/her diligence in reviewing the paper and for his/her helpful comments, which have improved the quality of the manuscript.

\subsection{Notation and Conventions}

Throughout the paper, $K$ denotes a field, and, when $p$ appears, it means that $K$ has positive characteristic $p$. The symbols $K_{{\rm{perf}}}$ and $K_s$ denote perfect and separable closures of $K$, respectively. For a scheme $S$, we regard the category of $S$-group schemes as a fully faithful subcategory of the category of fppf group sheaves on $S$. In particular, when we say that a sequence of $S$-group schemes is exact, or that a map of $S$-group schemes is surjective, et cetera, we mean that the corresponding property of fppf sheaves holds. Additionally, a group scheme is solvable when the corresponding fppf group sheaf is. Finally, a unipotent $K$-group scheme $U$ is {\em semiwound} when it does not admit a nonzero $K$-homomorphism from $\Ga$. This notion exhibits some of the same properties as woundness (wound = semiwound + smooth and connected). In particular, it is insensitive to separable field extension \cite[Prop.\,A.3]{rosasuniv}.

\section{Degree of imperfection of an extension}
\label{degimpsection}

In \cite{beckermaclane}, Becker and Maclane introduced for a finite purely inseparable extension $L/K$ of fields of characteristic $p > 0$ the quantity $r$ defined by the equality $[L: KL^p] = p^r$, and showed that $r$ is the minimum number of elements required to generate $L$ over $K$ \cite[Th.\,6]{beckermaclane}. This quantity will play a fundamental role in this work, so we begin by defining it for any finite field extension in characteristic $p$ and giving it a name. Later, we will extend the definition below to any finite reduced algebra over a field of characteristic $p$ (Definition \ref{degofprimdefalgebras}).

\begin{definition}
\label{degofprimfield}
For a finite extension $L/K$ of fields of characteristic $p > 0$, the {\em degree of imperfection} of $L/K$, denoted $\imp(L/K)$, is the non-negative integer $r$ defined by the equality $[L: KL^p] = p^r$, where $KL^p$ denotes the compositum inside $L$ of the subfields $K$ and $L^p$. Note that $L/KL^p$ is a purely inseparable extension of height $1$ (that is, $L^p \subset KL^p$), so $\imp(L/K) = {\rm{dim}}_L(\Omega^1_{L/KL^p}) = {\rm{dim}}_L(\Omega^1_{L/K})$.
\end{definition}

Our first result says that the degree of imperfection is insensitive to replacing the smaller field with an intermediate extension separable over it:

\begin{proposition}
\label{degfprimsepsubextn}
Given a tower $L/F/K$ of finite extensions of fields of characteristic $p > 0$, with $F/K$ separable, one has $\imp(L/K) = \imp(L/F)$.
\end{proposition}

\begin{proof}
Because $F/K$ is finite separable, one has $\Omega^1_{F/K} = 0$. Now use the exact sequence $$L \otimes_F \Omega^1_{F/K} \longrightarrow \Omega^1_{L/K} \rightarrow \Omega^1_{L/F} \longrightarrow 0.$$
\end{proof}

We also check that the degree of imperfection is insensitive to replacing $L$ and $K$ by their composita (in some larger field) with a separable extension of $K$:

\begin{proposition}
\label{degofprimcompwithsepble}
Let $L/K$ be a finite extension of fields of characteristic $p > 0$, and let $M$ be a field containing $L$, and $M/E/K$ a subfield separable $($but not necessarily algebraic$)$ over $K$. Then $\imp(L/K) = \imp(LE/E)$, where $LE$ is the compositum inside $M$ of $L$ and $E$.
\end{proposition}

\begin{proof}
Write $L/K$ as a tower $L/F/K$ with $L/F$ purely inseparable and $F/K$ separable. By Proposition \ref{degfprimsepsubextn}, $\imp(L/K) = \imp(L/F)$ and $\imp(LE/E) = \imp(LE/FE)$. Since $FE$ is separable over $F$, we may replace $K$ by $F$ and we are therefore free to assume that $L/K$ is purely inseparable.

We claim that the natural map $L \otimes_K E \rightarrow LE$ is an isomorphism. It is clearly surjective, so this is equivalent to checking that $L \otimes_K E$ is a field. Because $E/K$ is separable, $L \otimes_K E$ is a reduced, finite $E$-algebra, hence a finite product of fields. But $L/K$ is purely inseparable, so $L \otimes_K E$ is radicial over $E$, hence has connected spectrum, and is therefore a field.

Now let $n := \imp(L/K)$. Then $\Omega^1_{L/K}$ is a $K$-vector space of dimension $n[L: K]$. On the other hand, we have
\[
\Omega^1_{LE/E} \simeq \Omega^1_{L \otimes_K E}/E \simeq \Omega^1_{L/K}\otimes_K E,
\]
so $\Omega^1_{LE/E}$ is an $E$-vector space of dimension $n[L: K]$. Since $[L: K] = [L\otimes_K E: E] = [LE: E]$. we deduce that $\imp(LE/E) = {\rm{dim}}_{LE}(\Omega^1_{LE/E}) = \imp(L/K)$.
\end{proof}

\begin{lemma}
\label{compwithsepclosindofemb}
Let $L/K$ be a finite extension of fields, $\overline{K}$ an algebraic closure of $K$, and $K_s \subset \overline{K}$ the separable closure of $K$ inside $\overline{K}$. Let $\sigma, \tau\colon L \hookrightarrow \overline{K}$ be $K$-embeddings. Then $\sigma(L)K_s = \tau(L)K_s$.
\end{lemma}

\begin{proof}
Let $L/F/K$ be the maximal subextension separable over $K$. Then we have $L = L_1 \dots L_n$, where $L_i = F(\alpha_i^{1/p^{n_i}})$ for some $\alpha_i \in F$ and $n_i > 0$. Then $$\sigma(L)K_s = (\sigma(L_1)K_s)\dots(\sigma(L_n)K_s),$$ and similarly for $\tau$. So it suffices to show that $\sigma(L_i)K_s = \tau(L_i)K_s$. That is, we have reduced to the case in which $L = F(\alpha^{1/p^n})$ for some $\alpha \in F$ and $n > 0$. Because $F/K$ is separable, we have $\sigma(L)K_s = K_s(\sigma(\alpha)^{1/p^n})$ and similarly for $\tau$. By symmetry, therefore, it suffices to check that $\tau(\alpha)$ is a $p^n$th power in $K_s(\sigma(\alpha)^{1/p^n})$. Since $\sigma(\alpha)$ and $\tau(\alpha)$ are $\Gal(K_s/K)$-conjugate, we may choose a $K$-automorphism of $K_s(\sigma(\alpha)^{1/p^n})$ sending $\sigma(\alpha)$ to $\tau(\alpha)$. Since $\sigma(\alpha)$ is a $p^n$th power in $K_s(\sigma(\alpha)^{1/p^n})$, it follows that $\tau(\alpha)$ is as well.
\end{proof}

We now prove the following key result, which allows us to extend the definition of the degree of imperfection from finite field extensions of a field $K$ to arbitrary finite reduced $K$-algebras.

\begin{proposition}
\label{degofprimindofembeddings}
Let $L_1/K, \dots, L_n/K$ be finite field extensions of a field $K$ of characteristic $p > 0$. Let $E/K, F/K$ be field extensions, and suppose given $K$-embeddings $\sigma_i\colon L_i \hookrightarrow E$, $\tau_i\colon L_i \hookrightarrow F$ for $i = 1, \dots, n$. Then $\imp(\sigma_1(L_1)\dots\sigma_n(L_n)/K) = \imp(\tau_1(L_1)\dots \tau_n(L_n)/K)$, where $\sigma_1(L_1)\dots\sigma_n(L_n)$ denotes the compositum inside $E$ of the fields $\sigma_i(L_i)$ and similarly for $\tau_1(L_1)\dots\tau_n(L_n)$ inside $F$.
\end{proposition}

\begin{proof}
First, choosing a field $M$ admitting $K$-embeddings of both $E$ and $F$, we may replace both $E$ and $F$ by $M$ and thereby assume that $E = F = M$. Replacing $M$ with an algebraic closure, we may also assume that $M$ is algebraically closed. Let $K_s \subset M$ be the separable closure of $K$ inside $M$. Then, using Proposition \ref{degofprimcompwithsepble} for the first and last equalities below, and Lemma \ref{compwithsepclosindofemb} for the third equality, we have
\begin{align*}
\imp(\sigma_1(L_1)\dots\sigma_n(L_n)/K) & = \imp(\sigma_1(L_1)\dots\sigma_n(L_n)K_s/K_s) \\
& = \imp((\sigma_1(L_1)K_s)\dots(\sigma_n(L_n)K_s)/K_s) \\
& = \imp((\tau_1(L_1)K_s)\dots(\tau_n(L_n)K_s)/K_s) \\
& = \imp(\tau_1(L_1)\dots\tau_n(L_n)K_s/K_s) \\
& = \imp(\tau_1(L_1)\dots\tau_n(L_n)/K). \qedhere
\end{align*}
\end{proof}

We now define the degree of imperfection for finite reduced algebras over fields.

\begin{definition}
\label{degofprimdefalgebras}
For a field $K$ of characteristic $p > 0$, and a finite reduced $K$-algebra $A$, we define the {\em degree of imperfection} of $A/K$, denoted $\imp(A/K)$, as follows: choose a $K$-algebra isomorphism $A \simeq \prod_{i=1}^n L_i$, where each $L_i$ is a finite field extension of $K$. Choose a tuple $(E, \sigma_1, \dots, \sigma_n)$ consisting of a field $E$ and $K$-embeddings $\sigma_i\colon L_i \hookrightarrow E$. Then we define $\imp(A/K) := \imp(\sigma_1(L_1)\dots\sigma_n(L_n)/K)$.
\end{definition}

If we have two isomorphisms $\phi\colon A \xrightarrow{\sim} \prod_{i=1}^nL_i$ and $\psi\colon A \xrightarrow{\sim} \prod_{i=1}^m F_i$ as in the definition above, then $m = n$ and the isomorphism $\psi \circ \phi^{-1}$ corresponds to a permutation $f$ of $\{1, \dots, n\}$ and, for $1 \leq i \leq n$, $K$-isomorphisms $\phi_i\colon L_i \xrightarrow{\sim} F_{f(i)}$. In particular, if we have $K$-embeddings $\{\tau_i\}$ of the $F_i$ into some field $E'$, then post-composing with the isomorphism $\psi \circ \phi^{-1}$, we obtain $K$-embeddings $\{\tau_i'\}$ of the $L_i$ into $E'$, and $\tau_i(F_i) = \tau'_i(L_i)$. Thus, Definition \ref{degofprimdefalgebras} is well-defined by Proposition \ref{degofprimindofembeddings}.

Now we prove that having degree of imperfection $0$ has a natural interpretation.

\begin{proposition}
\label{degofprim0=sepble}
For a finite reduced algebra $A$ over a field $K$ of characteristic $p > 0$, one has $\imp(A/K) = 0$ if and only if $A$ is \'etale over $K$.
\end{proposition}

\begin{proof}
When $A = L$ is a field, \'etaleness is equivalent to unramifiedness -- that is, $\Omega^1_{L/K} = 0$. Thus we are done in this case. Now we treat the general case. Choose a $K$-algebra isomorphism $A \simeq \prod_{i=1}^n L_i$ for field extensions $L_i/K$. We may choose the $L_i$ to all lie in some field $E$ containing $K$. Then $\imp(A/K) = \imp(L_1\dots L_n/K)$. Thus $\imp(A/K) = 0$ if and only if $L_1\dots L_n$ is separable over $K$, which in turn holds precisely when each $L_i$ is separable over $K$, i.e., when $A/K$ is \'etale.
\end{proof}

Next we check that degree of imperfection is insensitive to injective \'etale morphisms.

\begin{proposition}
\label{degofprimsepblextnlargerfield}
Let $A \rightarrow B$ be an injective \'etale morphism of finite reduced algebras over a field $K$ of characteristic $p > 0$. Then $\imp(A/K) = \imp(B/K)$.
\end{proposition}

\begin{proof}
We first consider the case in which $A$ and $B$ are fields. That is, we have a tower of finite field extensions $L/F/K$ with $L/F$ separable, and we need to show that $\imp(L/K) = \imp(F/K)$. The following sequence is exact (left-exactness coming from the separability of $L/F$):
\[
0 \longrightarrow L \otimes_F \Omega^1_{F/K} \longrightarrow \Omega^1_{L/K} \longrightarrow \Omega^1_{L/F} \longrightarrow 0.
\]
Because $L/F$ is \'etale, $\Omega^1_{L/F} = 0$, so $\imp(L/K) = \imp(F/K)$.

Now consider the general case. Choose $K$-algebra isomorphisms $A \simeq \prod_{i=1}^n F_i$ and $B \simeq \prod_{j=1}^m L_j$ with the $F_i$ and $L_j$ field extensions of $K$ such that the $L_j$ are all contained in some field $E$ containing $K$. The \'etale $K$-morphism $\phi \colon A \rightarrow B$ corresponds to a partition $I_1 \coprod \dots \coprod I_n$ of $\{1, \dots, m\}$ and for each pair $(i, j)$ with $j \in I_i$ a $K$-embedding $\phi_{i,\,j} \colon F_i \rightarrow L_j$ making $L_j$ separable over $F_i$. The injectivity of $\phi$ exactly says that $I_i$ is nonempty for each $1 \leq i \leq n$. For each $1 \leq i \leq n$, choose some $j_i \in I_i$.

By definition,
\[
\imp(A/K) = \imp(\phi_{1,\,j_1}(F_1)\dots \phi_{n,\,j_n}(F_n)/K),
\]
and
\[
\imp(B/K) = \imp(L_1\dots L_m/K).
\]
Because $L_{j_i}$ is separable over $\phi_{i,\,j_i}(F_i)$, the extension $L_1\dots L_m/\phi_{1,\,j_1}(F_1)\dots \phi_{n,\,j_n}(F_n)$ is separable. The proposition therefore follows from the already-treated case when $A$ and $B$ are fields.
\end{proof}

The next result says that degree of imperfection cannot decrease upon passage to larger extensions.

\begin{proposition}
\label{degofprimjumpsoverlargerexts}
Give an injective $K$-homomorphism $A \rightarrow B$ of finite reduced algebras over a field $K$ of characteristic $p > 0$, one has $\imp(A/K) \leq \imp(B/K)$. In particular, given a tower $L/F/K$ of finite field extensions, one has $\imp(F/K) \leq \imp(L/K)$.
\end{proposition}

\begin{proof}
Choosing isomorphisms $A \simeq \prod_{i=1}^n F_i$ and $B \simeq \prod_{j=1}^m L_j$, with $F_i, L_j$ fields, a $K$-algebra homomorphism $f\colon A \rightarrow B$ is the same thing as a partition $I_1 \coprod \dots \coprod I_n$ of $\{1, \dots, m\}$ and for each $1 \leq i \leq n$ and each $j \in I_i$ a $K$-embedding $\phi_{i,\,j}\colon F_i \hookrightarrow L_j$. To say that $f$ is injective means that each $I_i$ is non-empty. In particular, if we choose a field $E$ and $K$-embeddings $\tau_j\colon L_j \hookrightarrow E$, then by arbitrarily choosing some $j_i \in I_i$ for each $i$, we obtain for each $1 \leq i \leq n$ the $K$-embedding $\sigma_i := \tau_{j_i}\circ \phi_{i,\, j_i}\colon F_i \hookrightarrow E$. Then $\sigma_i(F_i) \subset \tau_{j_i}(L_{j_i})$, so $\sigma_1(F_1)\dots\sigma_n(F_n) \subset \tau_1(L_1)\dots\tau_m(L_m)$. In particular, we see that the general case of the proposition follows from the case in which $A$ and $B$ are fields, so we concentrate on this case.

Suppose given a tower of finite extensions $L/F/K$. Then we claim that
\begin{equation}
\label{degofprimjumpsoverlargerextspfeqn1}
[KL^p: KF^p] \leq [L^p: F^p] = [L: F].
\end{equation}
Indeed, the first inequality follows from the general fact that, for any finite extension $E/F$, with $E$ contained in the field $M$, and any subfield $N \subset M$, one has $[EN: FN] \leq [E: F]$. The second equality follows from the fact that the $p$th power map induces an isomorphism from $L/F$ onto $L^p/F^p$. Now, using (\ref{degofprimjumpsoverlargerextspfeqn1}), we find that
\[
[L: KL^p] = \frac{[L: KF^p]}{[KL^p: KF^p]} \geq \frac{[L: KF^p]}{[L: F]} = [F: KF^p].
\]
That is, $\imp(L/K) \geq \imp(F/K)$.
\end{proof}

\begin{proposition}
\label{degimpleqdegimperf}
If $K$ is a field of degree of imperfection $r$, then for any finite reduced $K$-algebra $A$, one has $\imp(A/K) \leq r$.
\end{proposition}

\begin{proof}
By the very definition of $\imp(A/K)$, the assertion reduces to the case when $A = L$ is a field. By Proposition \ref{degfprimsepsubextn}, and since any finite extension of $K$ also has degree of imperfection $r$, we may assume that $L/K$ is purely inseparable. Then $L \subset K^{1/p^n}$ for some $n > 0$, so by Proposition \ref{degofprimjumpsoverlargerexts}, it suffices to show that $\imp(K^{1/p^n}/K) = r$. Let $K_m := K^{1/p^m}$. Then we have $KK_n^p = K_{n-1}$, so $[K_n: KK_n^p] = [K_n: K_{n-1}] = [K_{n-1}^{1/p}: K_{n-1}] = p^r$ because the finite extension $K_{n-1}$ of $K$ also has degree of imperfection $r$. Thus, $\imp(K_n/K) = r$.
\end{proof}

The degree of imperfection is invariant under separable base change.

\begin{proposition}
\label{degofprimsepblbasechange}
For a finite reduced algebra $A$ over a field $K$ of characteristic $p > 0$, and a $($not necessarily algebraic$)$ separable extension $K'/K$, $K' \otimes_K A$ is a finite reduced $K'$-algebra, and $$\imp(A/K) = \imp((K' \otimes_K A)/K').$$
\end{proposition}

\begin{proof}
The separability of $K'/K$ ensures that $K' \otimes_K A$ is still reduced \cite[Def.\,27.(D)]{matsumura}, so we turn to the assertion about the degree of imperfection. We first consider the case when $K' = K_s$ is a separable closure of $K$.
Choose an algebraically closed field $\overline{K}$ containing $K_s$, and choose a $K$-algebra isomorphism $A \simeq \prod_{i=1}^n L_i$, where the $L_i$ are subfields of $\overline{K}$. By definition, $\imp(A/K) = \imp(L_1\dots L_n/K)$, and by Proposition \ref{degofprimcompwithsepble}, we obtain
\begin{equation}
\label{degofprimsepblbasechangepfeqn2}
\imp(A/K) = \imp((L_1K_s)\dots(L_nK_s)/K_s).
\end{equation}
Choose for each $i$ a $K_s$-algebra isomorphism 
\begin{equation}
\label{degofprimsepblbasechangepfeqn7}
K_s \otimes_K L_i \simeq \prod_{j=1}^{n_i} M_{i,\,j}
\end{equation}
with the $M_{i,\,j}$ subfields of $\overline{K}$. The $K$-algebra homomorphism
\begin{equation}
\label{degofprimsepblbasechangepfeqn14}
L_i \xrightarrow{1 \otimes_K {\rm{id}}} K_s \otimes_K L_i \simeq \prod_{j=1}^{n_i} M_{i,\,j}
\end{equation}
shows that there exists for each $i, j$ a $K$-embedding $\sigma_{i,\,j} \colon L_i \hookrightarrow M_{i,\,j}$. On the one hand, each $M_{i,\,j}$ contains $K_s$. On the other hand, for any $K$ embedding $\sigma\colon L_i \hookrightarrow \overline{K}$, Proposition \ref{compwithsepclosindofemb} implies that $\sigma(L_i)K_s = L_iK_s$, so we see that $L_i \subset \sigma_{i,\,j}(L_i)K_s \subset M_{i,\,j}$. Thus, $L_iK_s \subset M_{i,\,j}$ for all $i, j$.

In fact, we claim that 
\begin{equation}
\label{degofprimsepblbasechangepfeqn8}
M_{i,\,j} = L_iK_s
\end{equation}
for all $i, j$. Since $L_iK_s$ is separably closed, it suffices to show that $M_{i,\,j}$ is separable over $L_iK_s$ and because $L_iK_s = \sigma(L_i)K_s$ for any $K$-embedding $\sigma\colon L_i \hookrightarrow \overline{K}$, it suffices to check that it is separable over $L_i$ for {\em some} choice of embedding $L_i \hookrightarrow M_{i,\,j}$. In particular, we may check this for the embedding obtained from the composition (\ref{degofprimsepblbasechangepfeqn14}). Then the isomorphism (\ref{degofprimsepblbasechangepfeqn7}), together with the condition \cite[Def.\,27.(D)]{matsumura}, implies that it suffices to show that $K_s \otimes_K L_i$ is a separable $L_i$-algebra, and this follows from the fact that $K_s$ is a separable $K$-algebra. This completes the proof of (\ref{degofprimsepblbasechangepfeqn8}).

Let $M$ denote the compositum inside $\overline{K}$ of all of the $M_{i,\,j}$, $1 \leq i \leq n$, $1 \leq j \leq n_i$. We now have the following equalities, where the first holds by definition and the second by (\ref{degofprimsepblbasechangepfeqn8}):
\begin{equation}
\label{degofprimsepblbasechangepfeqn3}
\imp(K_s \otimes_K A/K_s) = \imp(M/K_s) = \imp((L_1K_s)\dots(L_nK_s)/K_s).
\end{equation}
Combining (\ref{degofprimsepblbasechangepfeqn3}) and (\ref{degofprimsepblbasechangepfeqn2}) proves the proposition in the case that $K' = K_s$ is a separable closure of $K$.

Next we consider the case in which $K$ is separably closed. Choose an algebraically closed field $E$ containing $K'$, and choose a $K$-algebra isomorphism $A \simeq \prod_{i=1}^n L_i$, where each $L_i$ is a subfield of $E$. By definition,
\begin{equation}
\label{degofprimsepblbasechangepfeqn9}
\imp(A/K) = \imp(L_1\dots L_n/K).
\end{equation}
As $K'$-algebras, $K' \otimes_K A \simeq \prod_{i=1}^n (K' \otimes_K L_i)$. We claim that the natural $K'$-algebra map $K' \otimes_K L_i \rightarrow K'L_i$ is an isomorphism. It suffices to show that $K' \otimes_K L_i$ is a field. First, it is reduced because $K'/K$ is separable. It is also finite over $K'$ because $L_i$ is finite over $K$.  It follows that $K' \otimes_K L_i$ is isomorphic to a product of finitely many fields. Because $K$ is separably closed, the finite $K$-scheme ${\rm{Spec}}(L_i)$ is geometrically connected, hence ${\rm{Spec}}(K' \otimes_K L_i)$ is connected, so $K' \otimes_K L_i$ is a field. So, as claimed, we have the $K'$-algebra isomorphism $K' \otimes_K L_i \xrightarrow{\sim} K'L_i$. Therefore, by definition,
\begin{equation}
\label{degofprimsepblbasechangepfeqn10}
\imp((K' \otimes_K A)/K') = \imp((K'L_1)\dots(K'L_n)/K').
\end{equation}
Combining (\ref{degofprimsepblbasechangepfeqn9}), (\ref{degofprimsepblbasechangepfeqn10}), and Proposition \ref{degofprimcompwithsepble} completes the proof of the proposition when $K$ is separably closed.

Now we prove the proposition in general. Choose an algebraically closed field $E$ containing $K'$, and let $K_s, K'_s$ denote the separable closures of $K$ and $K'$, respectively, in $E$. By the already-treated case in which $K'$ is a separable closure of $K$, we have
\[
\imp(A/K) = \imp(K_s \otimes_K A/K_s)
\]
\begin{equation}
\label{degofprimsepblbasechangepfeqn11}
\imp((K' \otimes_K A)/K') = \imp(K'_s \otimes_K A/K'_s).
\end{equation}
Now applying the already-treated case when $K$ is separably closed shows that the right sides of (\ref{degofprimsepblbasechangepfeqn11}) agree.
\end{proof}

Becker and Maclane showed that, for a finite purely inseparable field extension $L/K$, the degree of imperfection measures the minimum number of elements required to generate $L$ over $K$ \cite[Th.\,6]{beckermaclane}. This almost always holds in general, even without the purely inseparable assumption.

\begin{proposition}
\label{degofprimminnumofgens}
For a finite extension $L/K$ of fields of characteristic $p > 0$, $\imp(L/K)$ equals the minimum number of generators required to generate $L$ over $K$, unless $L/K$ is separable and $L \neq K$ $($in which case, by Proposition $\ref{degofprim0=sepble}$, $\imp(L/K) = 0$, while the minimum number of generators is $1$ by the Primitive Element Theorem$)$.
\end{proposition}

\begin{proof}
The proposition holds when $L/K$ is purely inseparable \cite[Th.\,6]{beckermaclane}. We will show that (i) for $r \geq 0$, if $L$ is generated over $K$ by $r$ elements, then $\imp(L/K) \leq r$; and (ii) for $r > 0$, if $\imp(L/K) = r$, then $L$ is generated over $K$ by $r$ elements. Combining (i) and (ii) proves the proposition. Let $L/F/K$ be the maximal subextension separable over $K$, so $L/F$ is purely inseparable. First we prove (i). Since $L$ is generated over $K$ by $r$ elements, it is generated over $F$ by $r$ elements. Therefore, $\imp(L/F) \leq r$. By Proposition \ref{degfprimsepsubextn}, therefore, $\imp(L/K) \leq r$.

Next we prove (ii). Let $r > 0$,  and suppose that $\imp(L/K) = r$. Then, because $L/F$ is purely inseparable, $L$ is generated over $F$ by $r$ elements: $L = F(\alpha_1, \dots, \alpha_r)$ for some $\alpha_i \in L$. Since $F/K$ is finite separable, $F = K(\theta)$ for some $\theta \in F$; in particular, $\theta$ is separable over $K$. We will show by an argument essentially identical to one of the standard proofs of the Primitive Element Theorem that $L$ is generated over $K$ by $r$ elements. In fact, we will show that for all but finitely many $\lambda \in K$, we have $L = K(\alpha_1, \dots, \alpha_{r-1}, \alpha_r + \lambda\theta)$. Since we may assume $K$ to be infinite (since finite fields are perfect, so the proposition is trivial over them), this will prove the result.

Let $f(X), g(X) \in K[X]$ be the minimal polynomials of $\theta$ and $\alpha_r$, respectively, over $K$. Then we claim that, as long as $\lambda$ is not of the form $(\alpha' - \alpha_r)/(\theta - \theta')$ for some root $\alpha'$ of $g(X)$ and some root $\theta' \neq \theta$ of $f(X)$ (lying in some fixed algebraic closure of $L$), then $L = K(\alpha_1, \dots, \alpha_{r-1}, \alpha_r + \lambda\theta)$. Indeed, let $m(X) \in K(\alpha_1, \dots, \alpha_{r-1}, \alpha_r + \lambda\theta)[X]$ be the minimal polynomial of $\theta$ over $K(\alpha_1, \dots, \alpha_{r-1}, \alpha_r + \lambda\theta)$. We claim that ${\rm{deg}}(m) = 1$, so that $\theta \in K(\alpha_1, \dots, \alpha_{r-1}, \alpha_r + \lambda\theta)$, hence also $\alpha_r \in K(\alpha_1, \dots, \alpha_{r-1}, \alpha_r + \lambda\theta)$, so $L = K(\alpha_1, \dots, \alpha_{r-1}, \alpha_r + \lambda\theta)$.

To prove the claim, assume for the sake of contradiction that ${\rm{deg}}(m) > 1$. Since $\theta$ is separable over $K$, $f$ has no repeated roots, hence neither does $m$. Therefore, $m$ has a root $\theta' \neq \theta$. Further, $m(X) \mid f(X), g(\alpha_r + \lambda\theta - \lambda X)$ since $\theta$ is a root of both of these polynomials. So $\theta'$ is also a root of these two polynomials. It follows that
\[
\alpha_r + \lambda\theta - \lambda \theta' = \alpha'
\]
for some root $\alpha'$ of $g$. We therefore find that $\lambda = (\alpha' - \alpha_r)/(\theta - \theta')$, contrary to our choice of $\lambda$.
\end{proof}

\section{Permawound unipotent groups}
\label{asunivsection}

A central role in the proof of the rigidity theorem \ref{rigiditytheorem} is played by permawound groups, which were introduced in \cite{rosasuniv}. These groups will also make an appearance later when we discuss applications of rigidity to the study of such groups in \S\ref{appstoasunivsec}, so we recall some of their most important properties for the convenience of the reader. The definition of these groups was given in Definition \ref{asunivdef}.

Suppose that $F \in K[X_1, \dots, X_n]$ is a $p$-polynomial -- that is, a sum of terms of the form $cX_i^{p^n}$ with $c \in K$, $n \geq 0$. The {\em principal part} $P$ of $F$ is the sum of the monomials $c_iX_i^{p^{d_i}}$ of highest degrees in each of the variables $X_i$. We say that $F$ is reduced if $P$ has no zeroes in $K^n$ apart from the trivial zero $\vec{0} \in K^n$. We say that $P$ is {\em universal} if the homomorphism $K^n \rightarrow K$ induced by $P$ is surjective. If $U_F \subset \Ga^n$ is the unipotent $K$-group defined by the vanishing of $F$, then $U_F$ is smooth precisely when $F$ has nonzero linear part, in which case it is permawound if and only if $P$ is universal \cite[Th.\,6.10]{rosasuniv}.

Permawound groups over imperfect fields are connected \cite[Prop.\,6.2]{rosasuniv}, but their most important properties, and the ones we shall use, are their so-called genericity and rigidity properties (the latter term perhaps being unfortunately ambiguous given the topic of this paper). First we recall genericity.

\begin{theorem}$($``Genericity,'' $\cite[Th.\,1.4]{rosasuniv}$$)$
\label{genericity}
Let $K$ be a field of finite degree of imperfection. Then for any smooth commutative $p$-torsion semiwound unipotent $K$-group $U$, there is an exact sequence
\[
0 \longrightarrow U \longrightarrow W \longrightarrow V \longrightarrow 0
\]
with $W$ wound, commutative, $p$-torsion, and permawound, and $V$ a vector group.
\end{theorem}

The above result is called genericity because it says that very general unipotent groups may be embedded in permawound ones. The other result which, in conjunction with genericity, makes permawound groups useful is called rigidity, because it says that permawound groups have a very rigid structure. First we must recall a definition. For any separably closed field $K$ of finite degree of imperfection, we define a wound unipotent $K$-group $\mathscr{V}$ as follows. Take any $p$-basis $\lambda_1, \dots, \lambda_r$ of $K$, let $I$ denote the set of functions $\{1, \dots, r\} \rightarrow \{0, 1, \dots, p-1\}$, and let $\mathscr{V} \subset \Ga^{I}$ be the following $K$-group scheme:
\begin{equation}
\label{Vdef}
\mathscr{V} := \left\{-X_0 + \sum_{f\in I} \left(\prod_{i=1}^r\lambda_i^{f(i)}\right)X_f^p = 0\right\},
\end{equation}
where the subscript $0$ denotes the constant function with value $0$. The isomorphism class of $\mathscr{V}$ is independent of our choice of $p$-basis \cite[Cor.\,7.2]{rosasuniv}. The rigidity property of permawound groups is contained in the theorem below. As usual, $\alpha_p$ over an $\F_p$-scheme $S$ denotes the base change to $S$ of the $\F_p$-group scheme $\F_p[X]/(X^p)$ with group law $X\cdot Y := X + Y$, while $\R_{K^{1/p}/K}$ denotes Weil restriction of scalars from $K^{1/p}$ to $K$. The group $\mathscr{V}$ is permawound and semiwound, while $\R_{K^{1/p}/K}(\alpha_p)$ is semiwound and weakly permawound (the weakly because it is not smooth for imperfect $K$). Indeed, $\mathscr{V}$ is semiwound by its defining equation together with \cite[Lem.\,B.1.7,(1)$\implies$(2)]{cgp}, while $\R_{K^{1/p}/K}(\alpha_p)$ is semiwound by the functorial property defining Weil restriction together with the semiwoundness of $\alpha_p$. For the assertions about (weak) permawoundness, see \cite[Prop.\,7.6]{rosasuniv}. Recall that, when $K$ is imperfect, permawound groups are connected, so in fact $\mathscr{V}$ is a wound unipotent group in that case.

\begin{theorem}$($``Rigidity,'' $\cite[Th.\,1.5]{rosasuniv}$$)$
\label{rigidityasuniv}
Let $K$ be a separably closed field of finite degree of imperfection, and let $U$ be a wound permawound unipotent $K$-group scheme. Then $U$ admits a filtration $1 = U_0 \trianglelefteq U_1 \trianglelefteq \dots \trianglelefteq U_m = U$ such that, for each $1 \leq i \leq m$, either $U_i/U_{i-1} \simeq \mathscr{V}$ or $U_i/U_{i-1} \simeq \R_{K^{1/p}/K}(\alpha_p)$.
\end{theorem}

The utility of genericity and rigidity is that they reduce many problems about wound unipotent groups to the cases of the two groups $\R_{K^{1/p}/K}(\alpha_p)$ and $\mathscr{V}$ appearing in Theorem \ref{rigidityasuniv}. As an illustration that will play a central role in the proof of the rigidity theorem \ref{rigiditytheorem}, we prove the following general result.

\begin{proposition}
\label{containsallwound}
Let $K$ be a separably closed field of finite degree of imperfection, and let $\mathscr{F}$ be a collection of isomorphism classes of $K$-group schemes such that
\begin{itemize}
\item[(i)] $\mathscr{F}$ is closed under subgroups: If $G \in \mathscr{F}$ and $G' \subset G$ is a closed $K$-subgroup scheme, then $G' \in \mathscr{F}$.
\item[(ii)] $\mathscr{F}$ is closed under extensions: If one has an exact sequence of $K$-group schemes
\[
1 \longrightarrow G' \longrightarrow G \longrightarrow G'' \longrightarrow 1,
\]
and $G', G'' \in \mathscr{F}$, then $G \in \mathscr{F}$.
\item[(iii)] $\R_{K^{1/p}/K}(\alpha_p), \mathscr{V} \in \mathscr{F}$.
\end{itemize}
Then $\mathscr{F}$ contains every wound unipotent $K$-group.
\end{proposition}

\begin{proof}
Properties (iii) and (ii), in conjunction with the rigidity property of permawound groups (Theorem \ref{rigidityasuniv}), imply that $\mathscr{F}$ contains every wound permawound $K$-group. Property (i) together with the genericity property of permawound groups (Theorem \ref{genericity}) implies that $\mathscr{F}$ contains every commutative $p$-torsion wound unipotent $K$-group. Then property (ii) and \cite[Cor.\,B.3.3]{cgp} imply that $\mathscr{F}$ contains every wound unipotent $K$-group.
\end{proof}

The importance of Proposition \ref{containsallwound} for us is that the class $\mathscr{F}$ of groups $G$ satisfying the conclusion of the rigidity theorem \ref{rigiditytheorem} is easily seen to satisfy (i) and (ii). Thus to prove rigidity for wound unipotent groups -- the key case -- one only needs to verify it for the two groups $\R_{K^{1/p}/K}(\alpha_p)$ and $\mathscr{V}$ appearing in Proposition \ref{containsallwound} (once one carries out the easy reduction to the case when $K$ is separably closed of finite degree of imperfection). The former group trivially satisfies rigidity because it admits no nonzero maps from smooth (or even geometrically reduced) $K$-schemes. Thus we are reduced to the case $G = \mathscr{V}$. In order to prove rigidity in this special case (actually, we will use an inductive argument and reduce this case of rigidity to a case with smaller degree $r$ of imperfection), we need to carry out computations with certain moduli spaces of morphisms from rational curves into $\mathscr{V}$. This is the business of \S\ref{modcompsection}.

\section{A moduli space computation}
\label{modcompsection}

In this section we will carry out computations with certain moduli spaces which play an important role in the proof of the rigidity theorem. To set the stage, we recall the following result, which follows from \cite[Th.\,1.4]{rosmodulispaces}. Although we do not strictly speaking require it, it provides a useful jumping off point pedagogically. Recall that a unipotent group is called {\em semiwound} if it does not contain a copy of $\Ga$.

\begin{theorem}
\label{modspacesexist}
For a field $K$, a geometrically reduced finite type $K$-scheme $X$, a point $x \in X(k)$, and a semiwound unipotent $K$-group $U$, let $$\calMor((X, x), (U, 1))\colon \{K-{\mbox{schemes}}\} \rightarrow \{\mbox{groups}\}$$ denote the functor of pointed morphisms, defined by the formula $T \mapsto Mor_T((X_T, x_T), (U_T, 1_T))$. Then there is a unique subfunctor $\calMor((X, x), (U, 1))^+ \subset \calMor((X, x), (U, 1))$ with the following two properties:
\begin{itemize}
\item[(i)] The inclusion $\calMor((X, x), (U, 1))^+ \subset \calMor((X, x), (U, 1))$ is an equality on $T$-points for every geometrically reduced $K$-scheme $T$.
\item[(ii)] The functor $\calMor((X, x), (U, 1))^+$ is represented by a smooth $K$-group scheme.
\end{itemize}
Furthermore, the scheme described in (ii), which we also denote $\calMor((X, x), (U, 1))^+$, is semiwound unipotent.
\end{theorem}

We are interested in computations with these moduli spaces in certain special cases, when $U$ takes on certain special values and $(X, x) = (\P^1\backslash\{\lambda^{1/p^n}\}, \infty)$ for various $\lambda \in K$. In particular, the present section is devoted to the proof of Proposition \ref{modspacescomp} below, which gives a partial description of these moduli spaces in the cases in which we are interested.

\begin{proposition}
\label{modspacescomp}
Let $\lambda_1, \dots, \lambda_r \in K$ be a $p$-basis for $K$. For $d \geq 0$, let
$$I_d := \{f\colon \{1, \dots, r\} \rightarrow \{0, 1, \dots, p^d-1\}$$
and
$$J_d := \{f\in I_d \mid f \not\equiv 0\pmod{p} \mbox{ and } \exists i, 2 \leq i \leq r, \mbox{ with } f(i) \neq 0\},$$ and let $$U := \left\{X_0 = \sum_{f \in I_1}\left(\prod_{i=1}^r\lambda_i^{f(i)}\right)X_f^{p}\right\} \subset \Ga^{I_1}$$
and
$$W_d := \left\{X_0 = X_0^p + \sum_{f \in J_d}\left(\prod_{i=1}^r\lambda_i^{f(i)}\right)X_f^{p^d}\right\} \subset \Ga \times \Ga^{J_d}.$$
Then for any integer $s > 0$, the $K$-group scheme $$\calMor((\P^1\backslash\{\lambda_1^{1/p^{s}}\}, \infty), (U, 0))^+$$ embeds as a $K$-subgroup scheme of $W_s$.
\end{proposition}

The proof will occupy the remainder of the present section. We first note that, over a $K$-algebra $B$, a global section of $\P^1\backslash\{\lambda_1^{1/p^{s}}\}$ which vanishes at $\infty$ is simply a rational function of the form $G(T)/(T^{p^{s}} - \lambda_1)^c$ with $G \in B[T]$ of degree $< cp^{s}$, and by multiplying top and bottom by a power of $T^{p^{s}}-\lambda_1$, we may assume that $c$ is a power of $p$. Thus a pointed morphism $(\P^1\backslash\{\lambda_1^{1/p^{s}}\}, \infty) \rightarrow (U, 0)$ over $B$ is the same thing as a collection of rational functions 
\begin{equation}
\label{modspacescompeqn1}
X_f := G_f(T)/(T^{p^{s'}} - \lambda_1^{p^{s'-s}}),
\end{equation}
$f \in I_1$, with $s' \geq s$ an integer, each $G_f \in B[T]$ having degree $< p^{s'}$, and such that
\begin{equation}
\label{modspacescompeqn2}
X_0 = \sum_{f \in I_1}\left(\prod_{i=1}^r\lambda_i^{f(i)}\right)X_f^{p}.
\end{equation}
A rational function as in (\ref{modspacescompeqn1}) with ${\rm{deg}}(G_f) < p^{s'}$ may equivalently be described as a power series $X_f = \sum_{n \geq 0} a_{f,n}T^n$ with $a_{f,n} \in B$ satisfying
\begin{equation}
\label{modspacescompeqn3}
a_{f,n} = \lambda_1^{p^{s'-s}}a_{f, n + p^{s'}}.
\end{equation}
for all $n \geq 0$. We will describe pointed morphisms $$(\P^1\backslash\{\lambda_1^{1/p^{s}}\}, \infty) \rightarrow (U, 0)$$ using this interpretation.

We begin with the following lemma.

\begin{lemma}
\label{mins'}
When $B = K_s$, $(\ref{modspacescompeqn3})$ holds with $s' = s$.
\end{lemma}

\begin{proof}
Equating coefficients of $T^{pn}$ on both sides of (\ref{modspacescompeqn2}), one obtains
\begin{equation}
\label{modspacescompeqn4}
a_{0,pn} = \sum_{f \in I_1} \left(\prod_{i=1}^r\lambda_i^{f(i)}\right)a_{f,n}^{p}.
\end{equation}
If we let $s'$ be the minimal integer $\geq s$ such that (\ref{modspacescompeqn3}) holds, then we wish to show that $s' < s+1$. Assume for the sake of contradiction that $s' \geq s+1$. Replacing $n$ by $n + p^{s'-1}$ in (\ref{modspacescompeqn4}), and using (\ref{modspacescompeqn3}), we find that
\begin{equation}
\label{modspacescompeqn5}
a_{0,pn} = \lambda_1^{p^{s'-s}}\sum_{f \in I_1} \left(\prod_{i=1}^r\lambda_i^{f(i)} \right)a_{f,n+p^{s'-1}}^{p}.
\end{equation}
If $s' \geq s+1$, then $\lambda_1^{p^{s'-s}} \in K^p$, so comparing (\ref{modspacescompeqn4}) and (\ref{modspacescompeqn5}), and using the $p$-independence of the $\lambda_i$, we conclude that 
\[
a_{f,n} = \lambda_1^{p^{s'-s-1}}a_{f, n+p^{s'-1}}
\]
for all $n\geq 0$. This contradicts the minimality of $s'$, so we conclude that (\ref{modspacescompeqn3}) does indeed hold with $s' = s$ when $B = K_s$.
\end{proof}

Now set $s' := s$ for the remainder of this section, so that for $K_s$-points of $$\calMor((\P^1\backslash\{\lambda_1^{1/p^{s}}\}, \infty), (U, 0))^+,$$ we have
\begin{equation}
\label{modspacescompeqn25}
a_{f,n} = \lambda_1a_{f, n + p^{s}}.
\end{equation}
Suppose that $B = K_s$. For $d \geq 0$, recall that $I_d$ denotes the set of functions $\{1, \dots, r\} \rightarrow \{0, 1, \dots, p^d-1\}$. Because the $\lambda_i$ form a $p$-basis of $K$, hence also of $K_s$, there exist $b_g \in K_s$ such that
\begin{equation}
\label{modspacescompeqn9}
a_{0,0} = \sum_{g \in I_{s}} \left( \prod_{i=1}^r \lambda_i^{g(i)} \right)b_g^{p^{s}}.
\end{equation}
The equations (\ref{modspacescompeqn25}) (which expresses the fact that the $X_f$ form rational functions vanishing at $\infty$ with denominators dividing $T^{p^{s}} - \lambda_1$) and (\ref{modspacescompeqn2}) together are equivalent to the following three equations:
\begin{equation}
\label{modspacescompeqn6}
a_{f,n} = \lambda_1a_{f,n+p^{s}}
\end{equation}
\begin{equation}
\label{modspacescompeqn7}
a_{0,pn} = \sum_{f \in I_1}\left(\prod_{i=1}^r\lambda_i^{f(i)}\right)a_{f,n}^{p}
\end{equation}
\begin{equation}
\label{modspacescompeqn8}
a_{0,n} = 0, \mbox{ $p \nmid n$}.
\end{equation}
Applying (\ref{modspacescompeqn6}) repeatedly yields 
\begin{equation}
\label{modspacescompeqn20}
a_{0,p^{s}n} = \lambda_1^{-n}a_{0,0}.
\end{equation}
In conjunction with (\ref{modspacescompeqn9}), this allows us to write $a_{0, p^{s}n} = L^1_{0,p^{s}n}((b_g^{p^{s}})_g)$, where $L^1_{0, p^{s}n}$ is homogeneous linear over $K_s$ and depends only on $s, n$, and the $\lambda_i$ (that is, it is independent of the $a_{f, n}$ -- equivalently, of the particular $K_s$-map $\P^1\backslash\{\lambda_1^{1/p^{s}}\} \rightarrow U$). The $p$-independence of the $\lambda_i$ together with (\ref{modspacescompeqn7}) with $n$ replaced by $p^{s-1}n$ then allows us to similarly write $a_{0,p^{s-1}n}$ (with $p \nmid n$ to ensure that we have not already written $a_{0,n}$ as a polynomial in the $b_g$) as a homogeneous degree $p^{s-1}$ polynomial in the $b_g$, with the polynomial again not depending on the particular map $\P^1\backslash\{\lambda_1^{1/p^{s}}\} \rightarrow U$. Continuing in this manner, if $p^d \mid \mid n$, then we may use (\ref{modspacescompeqn7}) to inductively write $a_{0,n}$ as a homogeneous $p$-polynomial of degree $p^{\gamma}$ in the $b_g$, where $\gamma := \min\{d, s\}$. Now applying (\ref{modspacescompeqn7}) and the $p$-independence of the $\lambda_i$ again allows us to similarly write $a_{f,n}$ for $f \neq 0$ as a $p$-polynomial in the $b_g$, once again in a manner depending on $f$ and $n$, but not on the particular choice of morphism. In sum, we have shown how to write all of the $a_{f,n}$ as $p$-polynomials in the $b_g$ in a universal manner (i.e., with the $p$-polynomials not depending on the particular morphism) when $B = K_s$.

The infinite collection of equations (\ref{modspacescompeqn6})--(\ref{modspacescompeqn8}) then defines a $K$-group scheme $V$ in the $b_g$, and $V$ embeds naturally as a subfunctor of $\calMor((\P^1\backslash\{\lambda_1^{1/p^{s}}\}, \infty), (U, 0))$. Indeed, we use the universal polynomials above to write the $a_{f,n}$ in terms of the $b_g$ and the equations for $V$ ensure that these define a pointed morphism $(\P^1\backslash\{\lambda_1^{1/p^{s}}\}, \infty) \rightarrow (U, 0)$ over any $K$-algebra $B$. Furthermore, by construction (and Lemma \ref{mins'}), $V$ contains every $K_s$-point of $\calMor((\P^1\backslash\{\lambda_1^{1/p^{s}}\}, \infty), (U, 0))$, so that the maximal smooth $K$-subgroup scheme $V^{\sm}$ of $V$ equals $\calMor((\P^1\backslash\{\lambda_1^{1/p^{s}}\}, \infty), (U, 0))^+$:

\begin{lemma}
\label{V^sm=Morphisms}
$V^{\sm} \simeq \calMor((\P^1\backslash\{\lambda_1^{1/p^{s}}\}, \infty), (U, 0))^+$.
\end{lemma}

Before giving the next lemma, let us introduce a piece of terminology. Given an expression
\[
\alpha = \sum_{g \in I_d} \left(\prod_{i=1}^r\lambda_i^{g(i)}\right)c_g^{p^d}
\]
with all $c_g \in K_s$, its {\em $g$-component} is $c_g$. When $g = 0$, we also call this the {\em $p^d$-power part} of $\alpha$.

\begin{lemma}
\label{V^smsubsetW}
Let $U' \subset \Ga^{I_{s}}$ be the $K$-subgroup defined by the following equations:
\begin{equation}
\label{modspacescompeqn12}
b_{ph} = \sum_{\substack{g \in I_{s}\\g \equiv h \pmod{p^{s-1}}}} \left(\prod_{i=1}^r\lambda_i^{(g(i)-h(i))/p^{s-1}}\right)b_g^p, \mbox{ for all $h \in I_{s-1}$}.
\end{equation}
Then $V^{\sm} \subset U'$.
\end{lemma}

\begin{proof}
Because $V(K_s)$ is Zariski dense in $V^{\sm}$, it suffices to show that $V(K_s) \subset U'$. So we may assume that $b_g \in K_s$. The equations (\ref{modspacescompeqn9}) and (\ref{modspacescompeqn7}) (the latter with $n = 0$) together imply that $V$ sits naturally as a $K$-subgroup of the subgroup of $\Ga^{I_{s}}$ defined by the equation
\begin{equation}
\label{modspacescompeqn11}
\sum_{g \in I_{s}} \left( \prod_{i=1}^r \lambda_i^{g(i)} \right)b_g^{p^{s}} = \left(\sum_{g \in I_{s}} \left( \prod_{i=1}^r \lambda_i^{g(i)} \right)b_g^{p^{s}}\right)^{p} + 
\sum_{0 \neq f \in I_1} \left( \prod_{i=1}^r\lambda_i^{f(i)}\right)A_{f,0}((b_g)_g)^{p},
\end{equation}
with $A_{f,0}$ the $p$-polynomial in the $b_g$ defining $a_{f,n}$ as described above. Because $b_g \in K_s$, the $p$-independence of the $\lambda_i$ then implies, by taking $p$-power parts of both sides, that
\[
\sum_{\substack{g \in I_{s}\\g \equiv 0 \pmod{p}}} \left(\prod_{i=1}^r\lambda_i^{g(i)/p}\right)b_g^{p^{s-1}} = \sum_{g \in I_{s}} \left( \prod_{i=1}^r \lambda_i^{g(i)} \right)b_g^{p^{s}}.
\]
This equation may be written as
\[
\sum_{h \in I_{s-1}}\left(\prod_{i=1}^r\lambda_i^{h(i)}\right)b_{ph}^{p^{s-1}} = \sum_{h \in I_{s-1}}\left(\prod_{i=1}^r\lambda_i^{h(i)}\right)\left[\sum_{\substack{g \in I_{s}\\g\equiv h\pmod{p^{s-1}}}}\left(\prod_{i=1}^r\lambda_i^{(g(i)-h(i))/p^{s-1}}\right)b_g^p\right]^{p^{s-1}}
\]
Using the $p$-independence of the $\lambda_i$ once again, and taking $h$-components, one obtains
\[
b_{ph} = \sum_{\substack{g \in I_{s}\\g \equiv h \pmod{p^{s-1}}}} \left(\prod_{i=1}^r\lambda_i^{(g(i)-h(i))/p^{s-1}}\right)b_g^p, \mbox{ for all $h \in I_{s-1}$}.
\]
That is, $(b_g)_g$ defines a $K_s$-point of $U'$.
\end{proof}

\begin{lemma}
\label{eqnsforW}
One has for all $(b_g)_g \in U'$
\begin{equation}
\label{modspacescompeqn18}
b_{p^dh} = \sum_{\substack{g \in I_{s}\\g \equiv h\pmod{p^{s-d}}}}\left(\prod_{i=1}^r\lambda_i^{(g(i)-h(i))/p^{s-d}}\right)b_g^{p^d}, \hspace{.3 in} 0 \leq d \leq s, \hspace{.1 in} h \in I_{s-d}.
\end{equation}
\end{lemma}

\begin{proof}
The proof is by induction on $d$. For $d = 0$ the assertion is immediate, so suppose that $0 < d \leq s$ and that the assertion is known for $d-1$. We have
\begin{align*}
b_{p^dh} & = \sum_{\substack{g \in I_{s}\\g\equiv p^{d-1}h\pmod{p^{s-1}}}}\left(\prod_{i=1}^r\lambda_i^{(g(i)-p^{d-1}h(i))/p^{s-1}}\right)b_g^p \hspace{.3 in} \mbox{by (\ref{modspacescompeqn12}) with $h$ replaced by $p^{d-1}h$} \\
&= \sum_{\substack{g \in I_{s}\\g\equiv p^{d-1}h\pmod{p^{s-1}}}}\left(\prod_{i=1}^r\lambda_i^{(g(i)-p^{d-1}h(i))/p^{s-1}}\right) \\
\times & \sum_{\substack{\alpha \in I_{s}\\ \alpha\equiv g/p^{d-1}\pmod{p^{s-(d-1)}}}}\left(\prod_{i=1}^r\lambda_i^{(p\alpha(i)-g(i)/p^{d-2})/p^{s-(d-1)}}\right)b_{\alpha}^{p^d}
\end{align*}
\begin{align*}
& \mbox{by induction: (\ref{modspacescompeqn18}) applied with $(h, d)$ replaced by $(g/p^{d-1}, d-1)$}\\
&= \sum_{\substack{g, \alpha \in I_{s}\\g \equiv p^{d-1}h\pmod{p^{s-1}}\\ \alpha \equiv g/p^{d-1}\pmod{p^{s-d+1}}}} \left(\prod_{i=1}^r\lambda_i^{(\alpha(i)-h(i))/p^{s-d}}\right)b_{\alpha}^{p^d}.
\end{align*}
The map $(g, \alpha) \mapsto \alpha$ induces a bijection between pairs $(g, \alpha)$ as in the above sum, on the one hand, and those $\alpha \in I_{s}$ which are $\equiv h\pmod{p^{s-d}}$ on the other. Thus the induction is complete and the lemma is proved.
\end{proof}

Let $L_{s} \subset I_{s}$ denote those functions that are nonzero modulo $p$, and consider the map $\phi\colon U' \rightarrow \Ga \times \Ga^{L_{s}}$ defined by the formula $(b_g)_g \mapsto (b_0) \times (b_g)_{g \in L_{s}}$. 

\begin{lemma}
\label{phiinj}
The $K$-homomorphism $\phi$ is injective.
\end{lemma}

\begin{proof}
If $b_0 = 0$ and $b_g = 0$ for all $g \in L_{s}$, then given another $f \in I_{s}$, we may write $f = p^dh$ for some $0 < d < s$ and $h \in I_{s-d}$ with $h \not\equiv 0\pmod{p}$. Then (\ref{modspacescompeqn18}) yields $b_f = b_{p^dh} = 0$.
\end{proof}

Next we show that $\phi(U')$ maps into a certain subgroup of $\Ga \times \Ga^{L_{s}}$.

\begin{lemma}
\label{imageofW}
$\phi(U') \subset \Ga \times \Ga^{L_{s}}$ is contained in the subgroup defined by the equation
\[
b_0 = b_0^p + \sum_{\substack{g \in I_{s}\\g\not\equiv 0\pmod{p}}}\left(\prod_{i=1}^r\lambda_i^{g(i)}\right)b_g^{p^{s}}.
\]
\end{lemma}

\begin{proof}
We compute that for $(b_g)_g \in U'$,
\begin{align*}
b_0 & = \sum_{\substack{g\in I_{s}\\g\equiv 0\pmod{p^{s-1}}}}\left(\prod_{i=1}^r\lambda_i^{g(i)/p^{s-1}}\right)b_g^p \hspace{.3 in} \mbox{by (\ref{modspacescompeqn12}) with $h = 0$}\\
&= b_0^p + \sum_{0 \neq h \in I_1} \left(\prod_{i=1}^r\lambda_i^{h(i)}\right)b_{p^{s-1}h}^p \\
&= b_0^p + \sum_{0 \neq h \in I_1} \left(\prod_{i=1}^r\lambda_i^{h(i)}\right) \sum_{\substack{g \in I_{s}\\g \equiv h\pmod{p}}}\left(\prod_{i=1}^r\lambda_i^{(g(i)-h(i))}\right)b_g^{p^{s}} \hspace{.3 in} \mbox{(by (\ref{modspacescompeqn18}) with $d = s-1$)} \\
&= b_0^p + \sum_{\substack{g\in I_{s}, 0\neq h \in I_1\\g \equiv h\pmod{p}}}\left(\prod_{i=1}^r\lambda_i^{g(i)}\right)b_g^{p^{s}}.
\end{align*}
The map sending $(g, h)$ as in the above sum to $g$ yields a bijection between the index set and those $g \in I_{s}$ not divisible by $p$, so we obtain the desired equation on $\phi(U')$.
\end{proof}

\begin{lemma}
\label{pexpa0n}
For every $K$-algebra $B$, and every $B$-point of $V$, there exist for every $n,e \geq 0$ elements $c_{f,d,n} \in B$ such that
\begin{equation}
\label{modspacescompeqn13}
a_{0,n} = a_{0,n/p^e}^{p^e} + \sum_{d=1}^e\sum_{0 \neq f \in I_1}\left( \prod_{i=1}^r\lambda_i^{p^{d-1}f(i)}\right)c_{f,d,n}^{p^{d}},
\end{equation}
where we set $a_{0,m} = 0$ if $m \notin\Z$.
\end{lemma}

\begin{proof}
We proceed by induction on $e$. The $e = 0$ case is immediate. For the inductive step, we first claim that, for all $n \geq 0$, there exist $c_{f,n} \in K$ such that
\begin{equation}
\label{modspacescompeqn21}
a_{0,n} = a_{0,n/p}^p + \sum_{0 \neq f \in I_1}\left(\prod_{i=1}^r\lambda_i^{f(i)}\right)c_{f,n}^{p}.
\end{equation}
Indeed, this follows from (\ref{modspacescompeqn7}) and (\ref{modspacescompeqn8}). Now beginning with (\ref{modspacescompeqn13}) for $e$, we apply (\ref{modspacescompeqn21}) with $n$ replaced by $n/p^e$ to obtain (\ref{modspacescompeqn13}) for $e+1$ (with $c_{f,e+1,n} := c_{f,n/p^e}$). This completes the proof of the lemma.
\end{proof}

\begin{lemma}
\label{mostbgvanish}
Suppose that $(b_g)_g \in V(K_s)$, and let $g_0 \in I_{s}$. If $g_0(i) = 0$ for all $2 \leq i \leq r$ but $g_0 \not\equiv 0\pmod{p}$, then $b_{g_0} = 0$.
\end{lemma}

\begin{proof}
We have
\begin{align}
\label{modspacescompeqn22}
& \lambda_1^{-g_0(1)}\sum_{g\in I_{s}}\left(\prod_{i=1}^r\lambda_i^{g(i)}\right)b_{g}^{p^{s}} \nonumber \\
& = a_{0,p^{s}g_0(1)} \hspace{.3 in} \mbox{(\ref{modspacescompeqn20} with $n = g_0(1)$ and (\ref{modspacescompeqn9})} \nonumber \\
& = a_{0,g_0(1)}^{p^{s}} + \sum_{d=1}^s\sum_{0 \neq f \in I_1}\left(\prod_{i=1}^r\lambda_i^{p^{d-1}f(i)}\right)c_{f,d,p^{s}g_0(1)}^{p^{s}} \hspace{.3 in} \mbox{(\ref{modspacescompeqn13}) with $n = p^{s}g_0(1)$, $e = s$}.
\end{align}
Note that each term in the final sum has nonvanishing $p^s$-power part, because $f \in I_1$ implies $f \not\equiv 0\pmod{p}$. Because $g_0(1) \not\equiv 0\pmod{p}$ by assumption, (\ref{modspacescompeqn8}) implies that $a_{0,g_0(1)} = 0$, so the $p^s$-power part of the final expression in (\ref{modspacescompeqn22}) vanishes. Hence the $p^s$-power part of the left side does as well. This $p^s$-power part is none other than $b_{g_0}$, so $b_{g_0} = 0$.
\end{proof}

\begin{proof}[Proof of Proposition $\ref{modspacescomp}$]
Lemmas \ref{imageofW} and \ref{mostbgvanish} together imply that $\phi$ maps $V(K_s)$ into $W_s$. Further, $\phi$ is injective by Lemma \ref{phiinj}. Because $V(K_s)$ is Zariski dense in $V^{\sm}$, it follows that $V^{\sm}$ embeds as a $K$-subgroup scheme of $W_s$. We are therefore done by Lemma \ref{V^sm=Morphisms}.
\end{proof}

\section{Rigidity for almost complete rational curves} 
\label{rigalmostcompleteratlsection}

In this section we prove the rigidity theorem \ref{rigiditytheorem} when -- in the notation of the theorem statement -- each of the smooth projective curves $\overline{X}_i$ is $\P^1_K$ and each $D_i$ consists of a single closed point. This is the key case; the general case will be deduced from this one in \S\ref{deductionsection} using geometric class field theory. We begin with some lemmas.

\begin{lemma}
\label{givengensgenerate}
Let $L/K$ be a finite purely inseparable extension of fields with $\imp(L/K) = r$. If $S$ is a set of elements of $L$ which generate $L$ over $K$, then there exist $r$ elements of $S$ which generate $L$ over $K$.
\end{lemma}

\begin{proof}
In the case that $L/K$ has height one -- that is, $L^p \subset K$ -- this follows from \cite[Ch.\,V, \S13, no.\,1, Prop.\,1]{bourbakialgebra}. In general, consider the intermediate field $L^pK$, the compositum of $L^p$ and $K$ inside $L$. Then $L/L^pK$ is a purely inseparable extension of height one. Furthermore, by the definition of degree of imperfection, one has $\imp(L/K) = \imp(L/L^pK)$. Thus there are $r$ elements $\alpha_1, \dots, \alpha_r \in S$ which generate $L/L^pK$. We claim that the $\alpha_i$ generate the extension $L/L^{p^n}K$ for every $n \geq 0$. Since $L/K$ is purely inseparable, one has $L^{p^m} \subset K$ for some $m$, so this will prove the lemma.

We prove the claim by induction on $n$, the $n = 0$ case being trivial. So let $n > 0$ and suppose that $L/L^{p^{n-1}}K$ is generated by the $\alpha_i$. By the already-treated $n = 1$ case, we have $L \subset L^pK(\alpha_1, \dots, \alpha_r)$, so $L^{p^{n-1}} \subset L^{p^n}K(\alpha_1, \dots, \alpha_r)$, and therefore $L \subset L^{p^{n-1}}K(\alpha_1, \dots, \alpha_r) \subset L^{p^n}K(\alpha_1, \dots, \alpha_r)$.
\end{proof}

We will also require the following lemma regarding closed points on curves.

\begin{lemma}
\label{resfldclpt}
For a smooth curve $X$ over a field $K$ and a closed point $x \in X$, the residue field $K(x)$ is a primitive extension of $K$. That is, it is generated by a single element.
\end{lemma}

\begin{proof}
When $X = \A^1_K$, $x$ corresponds to some maximal ideal $m$ of $K[T]$, and the field $K(x) = K[T]/m$ is manifestly primitive, generated by the image of $T$. In general, one has an \'etale $K$-morphism $f\colon U \rightarrow \A^1_K$ for some open neighborhood $U$ of $x$. If $y := f(x)$, then $K(x)$ is finite separable over $K(y)$. Proposition \ref{degofprimsepblextnlargerfield} then implies that $\imp(K(x)/K) = \imp(K(y)/K)$. By the already-treated $\A^1$ case and Proposition \ref{degofprimminnumofgens}, this last quantity is $\leq 1$, and one more application of Proposition \ref{degofprimminnumofgens} then implies that $K(x)/K$ is primitive.
\end{proof}

Finally, we require one more lemma. Before stating it, we introduce some notation and terminology. We say that a locally finite type $K$-group scheme $G$ is {\em totally nonsmooth} when its maximal smooth $K$-subgroup scheme is trivial. Equivalently, $G(K_s) = 1$; see \cite[Lem.\,C.4.1, Rem.\,C.4.2]{cgp}. We should remark that total nonsmoothness only behaves well under passage to separable extensions. That is, if $G/K$ is totally nonsmooth, then $G_L$ is still totally nonsmooth over $L$ for any {\em separable} extension $L/K$, but not for arbitrary extensions in general. For instance, if $K$ is an imperfect field of characteristic $p$, and $a \in K - K^p$, then the $K$-group scheme $U := \{Y^p - aX^p = 0\} \subset \Ga^2$ is totally nonsmooth because $a \notin K_s^p$. But $U_{K^{1/p}} \simeq \Ga \times \alpha_p$ via the map $(X, Y) \mapsto (X, Y - a^{1/p}X)$, hence $(U_{K^{1/p}})^{\rm{sm}} \simeq (\Ga)_{K^{1/p}}$.

For a field $K$ of characteristic $p > 0$ and $\lambda \in K$, let $K(\lambda^{1/p^{\infty}})$ denote the union (in some perfect closure) of the fields $K(\lambda^{1/p^n})$ for $n > 0$. Additionally, for a smooth unipotent $K$-group scheme $U$, let $U_{\Split}$ denote the maximal split unipotent $K$-subgroup scheme. Then $U_{\Split} \trianglelefteq U$ (we will only use the case of commutative $U$ below), and $U/U_{\Split}$ is $K$-semiwound. The key to the proof of the special case of the rigidity theorem treated in this section is the following lemma.

\begin{lemma}
\label{maxlsplitU'}
Let $K$ be a separably closed field of finite degree of imperfection, and let $\lambda \in K-K^p$. For $s > 0$, let $V_s$ denote the smooth unipotent $K$-group $\calMor((\P^1\backslash{\lambda^{1/p^s}}, \infty), (\mathscr{V}, 0))^+$. Then there is a totally nonsmooth $K$-subgroup scheme $N \subset V_s$ such that $$\left((V_s)_{K(\lambda_1^{1/p^{\infty}})}\right)_{\Split} \subset N_{K(\lambda_1^{1/p^{\infty}})}.$$
\end{lemma}

\begin{proof}
If we replace $V_s$ by a larger group $V_s'$, and find a totally nonsmooth $N' \subset V_s'$ that makes the conclusion of the lemma hold for the pair $(V_s', N')$, then by taking $N := N' \cap V_s$, we obtain the lemma. We therefore apply Proposition \ref{modspacescomp} (and use the notation of that proposition), first completing $\lambda$ to a $p$-basis $\lambda = \lambda_1, \lambda_2, \dots, \lambda_m$ for $K$. Then we take $V_s' := W_s$, where recall that $W_s$ is described by the following equation:
\begin{equation}
\label{maxlsplitU'pfeqn7}
X_0 = X_0^p + \sum_{f \in J_s}\left(\prod_{i=1}^m\lambda_i^{f(i)}\right)X_f^{p^s}.
\end{equation}
Thus we seek a totally nonsmooth $K$-subgroup $N'$ of $W_s$ containing the maximal split subgroup of $W_s$ over the {\em inseparable} extension $K(\lambda_1^{1/p^\infty})$.

For $f \in J_s$, let $f' := f|\{2, 3, \dots, m\}$, and let ${\rm{ord}}_p(f')$ denote the maximum integer $t$ such that $p^t\mid f'$. Because $f' \neq 0$ by the definition of $J_s$, one has
\begin{equation}
\label{maxlsplitU'pfeqn2}
{\rm{ord}}_p(f') \leq s-1
\end{equation}
for all $f \in J_s$. Over the field $K(\lambda_1^{1/p^{\infty}})$, make the invertible change of variables on $W_s$ given by $X_f \mapsto X_f$ for $f \in J_s$, but
\begin{equation}
\label{maxlsplitU'pfeqn4}
Y_0 := X_0 + \sum_{f \in J_s}\sum_{d=1}^{{\rm{ord}}_p(f')}\left(\prod_{i=1}^m\lambda_i^{f(i)/p^d}\right)X_f^{p^{s-d}}.
\end{equation}
We compute the effect of this transformation on the equation for $W_s$:
\begin{align*}
0 &= \sum_{f \in J_s}\left(\prod_{i=1}^m\lambda_i^{f(i)}\right)X_f^{p^s} + X_0^p-X_0 \\
&= \sum_{f \in J_s}\left(\prod_{i=1}^m\lambda_i^{f(i)}\right)X_f^{p^s} + \left[Y_0 - \sum_{f \in J_s}\sum_{d=1}^{{\rm{ord}}_p(f')}\left(\prod_{i=1}^m\lambda_i^{f(i)/p^d}\right)X_f^{p^{s-d}}\right]^p \\
&- \left[Y_0 - \sum_{f \in J_s}\sum_{d=1}^{{\rm{ord}}_p(f')}\left(\prod_{i=1}^m\lambda_i^{f(i)/p^d}\right)X_f^{p^{s-d}}\right] \\
&= \sum_{f \in J_s}\left(\prod_{i=1}^m\lambda_i^{f(i)}\right)X_f^{p^s} + Y_0^p - Y_0 \\
&- \sum_{f \in J_s}\sum_{d=1}^{{\rm{ord}}_p(f')}\left(\prod_{i=1}^m\lambda_i^{f(i)/p^{d-1}}\right)X_f^{p^{s-d+1}} + \sum_{f \in J_s}\sum_{d=1}^{{\rm{ord}}_p(f')}\left(\prod_{i=1}^m\lambda_i^{f(i)/p^d}\right)X_f^{p^{s-d}} \\
&= \sum_{f \in J_s}\left(\prod_{i=1}^m\lambda_i^{f(i)}\right)X_f^{p^s} + Y_0^p - Y_0 \\
&- \sum_{f \in J_s}\sum_{d=0}^{{\rm{ord}}_p(f')-1}\left(\prod_{i=1}^m\lambda_i^{f(i)/p^d}\right)X_f^{p^{s-d}} + \sum_{f \in J_s}\sum_{d=1}^{{\rm{ord}}_p(f')}\left(\prod_{i=1}^m\lambda_i^{f(i)/p^d}\right)X_f^{p^{s-d}}.
\end{align*}
Most of the terms of the final two sums cancel, and we conclude that $W_s$ in these new coordinates is described by the equation
\[
Y_0 = Y_0^p + \sum_{f \in J_s}\left(\prod_{i=1}^m\lambda_i^{f(i)/p^{{\rm{ord}}_p(f')}}\right)X_f^{p^{s-{\rm{ord}}_p(f')}}.
\]
We claim that the maximal split $K(\lambda^{1/p^\infty})$-subgroup scheme of $(W_s)_{K(\lambda^{1/p^\infty})}$ is contained in $\{Y_0 = 0\}$. Indeed, suppose given a $K(\lambda^{1/p^{\infty}})$-morphism $f\colon \A^1 \rightarrow W_{s,K(\lambda^{1/p^\infty})}$ such that $f(0) = 0$. We must show that the projection of $f$ to the $Y_0$ coordinate vanishes. The map $f$ is given by polynomials $Y_0(T), X_f(T) \in K(\lambda_1^{1/p^{\infty}})[T]$ with vanishing constant terms such that
\begin{equation}
\label{maxlsplitU'pfeqn1}
Y_0(T) = Y_0(T)^p + \sum_{f \in J_s}\left(\prod_{i=1}^m\lambda_i^{f(i)/p^{{\rm{ord}}_p(f')}}\right)X_f(T)^{p^{s-{\rm{ord}}_p(f')}}.
\end{equation}
Assume for the sake of contradiction that $Y_0(T) \neq 0$, and let $D > 0$ be its degree. Let $c_0 \neq 0$ be the leading coefficient of $Y_0(T)$ and for $f \in J_s$, let $c_f$ be the coefficient of $T^{Dp^{1-s+{\rm{ord}}_p(f')}}$ in $X_f(T)$, where we take $c_f = 0$ if the exponent is not an integer. Then comparing coefficients of $T^{pD}$ on both sides of (\ref{maxlsplitU'pfeqn1}) yields
\begin{equation}
\label{maxlsplitU'pfeqn3}
0 = c_0^p + \sum_{f \in J_s}\left(\prod_{i=1}^m\lambda_i^{f(i)/p^{{\rm{ord}}_p(f')}}\right)c_f^{p^{s-{\rm{ord}}_p(f')}}.
\end{equation}
For each $f \in J_s$, there is by definition some $2 \leq i \leq m$ such that $p\nmid f(i)/p^{{\rm{ord}}_p(f')}$. In conjunction with (\ref{maxlsplitU'pfeqn2}) and the fact that $\lambda_2, \dots, \lambda_m$ are $p$-independent over $K(\lambda_1^{1/p^\infty})$, this implies that $c_0 = 0$, a contradiction. We deduce that $Y_0(T) = 0$, so that the maximal split $K(\lambda^{1/p^\infty})$-subgroup of $(W_s)_{K(\lambda^{1/p^\infty})}$ is contained in $\{Y_0 = 0\}$, as claimed.

Translating back to the original $X_0, X_f$ coordinates via (\ref{maxlsplitU'pfeqn4}), we conclude that the group $(W_{s, K(\lambda^{1/p^\infty})})_{\rm{split}}$ is contained in the subgroup
\[
X_0 + \sum_{f \in J_s}\sum_{d=1}^{{\rm{ord}}_p(f')}\left(\prod_{i=1}^m\lambda_i^{f(i)/p^d}\right)X_f^{p^{s-d}} = 0.
\]
For $s > 1$, this is not a subgroup defined over $K$, however. In order to obtain such a subgroup, we raise both sides to the $p^{s-1}$ to obtain the subgroup $N'$ of $W_s$ described by the following equation (in addition to the equation for $W_s$):
\begin{equation}
\label{maxlsplitU'pfeqn6}
X_0^{p^{s-1}} + \sum_{f \in J_s}\sum_{d=1}^{{\rm{ord}}_p(f')}\left(\prod_{i=1}^m\lambda_i^{f(i)p^{s-1-d}}\right)X_f^{p^{2s-1-d}} = 0.
\end{equation}
This is {\em defined over $K$} thanks to (\ref{maxlsplitU'pfeqn2}), and contains the maximal split subgroup over the extension $K(\lambda_1^{1/p^\infty})$.

It only remains to verify that $N'$ is a totally nonsmooth $K$-group scheme. Suppose given a point $\vec{x} := (x_0)\times (x_f)_f \in N'(K_s)$. We must check that $\vec{x} = 0$. If $s = 1$, then (\ref{maxlsplitU'pfeqn2}) implies that the sum above is empty, so we conclude that $x_0 = 0$. On the other hand, suppose that $s > 1$. In the equation (\ref{maxlsplitU'pfeqn6}) for $N'$, all of the exponents are $\geq s-1$ by (\ref{maxlsplitU'pfeqn2}), and in the product over $i$ appearing in each term in the sum, there is some $i$ (depending on $f$) such that $p\nmid f(i)$ (by the definition of $J_s$), so there is some $i$ such that the exponent of $\lambda_i$ is not divisible by $p^{s-1}$. The $p$-independence of the $\lambda_i$ over $K_s$ therefore implies that $x_0 = 0$ in this case as well. In both cases, using the equation (\ref{maxlsplitU'pfeqn7}) for $W_s$ and the $p$-independence of the $\lambda_i$ once more, we then conclude that $x_f = 0$ for all $f \in J_s$. That is, $\vec{x} = 0$.
\end{proof}

We now prove rigidity for rational curves that are ``almost complete.''

\begin{proposition}
\label{rigalmostcompleteratl}
Theorem $\ref{rigiditytheorem}$ holds when $G = U$ is wound unipotent and, for all $i$, $\overline{X}_i = \P^1_K$ and $D_i$ consists of a single geometrically irreducible closed point.
\end{proposition}

\begin{proof}
The proof proceeds by induction on $r$. First assume that $r = 0$. Because $U$ remains wound over $K_s$ \cite[Prop.\,B.3.2]{cgp} and $\imp(D/K)$ is invariant under base change to $K_s$ by Proposition \ref{degofprimsepblbasechange}, we are free to extend scalars and thereby assume that $K$ is separably closed. This means that $(D_i)_{\red}$ is (the spectrum of) a separable extension field of $K$ by Proposition \ref{degofprim0=sepble}. Thus, because $K = K_s$, each $D_i$ is a $K$-point of $\P^1$, hence each $X_i$ is $K$-isomorphic to $\A^1_K$. If we let $Y := \prod_{i=2}^n X_i$, then each $y \in Y(K)$ induces a $K$-morphism $f_y\colon X_1 \simeq \A^1_K \rightarrow U$ sending $x_1$ to $0$. Because $U$ is wound, the only such map is the $0$ map, so $f_y$ vanishes for each $y \in Y(K)$. Because $Y(K)$ is Zariski dense in $Y$, we deduce that $f = 0$, as desired.

Now suppose that $r > 0$ and that the assertion of the proposition holds for values less than $r$. By standard spreading out arguments, we may assume that $K$ is a finitely generated extension field of $\F_p$, and in particular of finite degree of imperfection. We may then extend scalars to $K_s$ and thereby assume that $K$ is separably closed of finite degree of imperfection (but no longer finitely generated). By a change of variables, we may also assume that $x_i = \infty$ for all $i$. Now fix $X := \prod_{i=1}^n X_i$. We wish to show that the constant map is the only map from $X$ into some particular $K$-group scheme $H$ whose restriction to a certain subscheme of $X$ is the constant map to $1$. Let $\mathscr{F}$ be the collection of (isomorphism classes of) $K$-group schemes $H$ with this property. Then $\mathscr{F}$ is closed under the formation of closed subgroup schemes and extensions. By Proposition \ref{containsallwound}, therefore, in order to show that $\mathscr{F}$ contains all wound unipotent groups, it is enough to show that it contains $\R_{K^{1/p}/K}(\alpha_p)$ and $\mathscr{V}$. That it contains the former group follows from the fact that $X$ is smooth while $\R_{K^{1/p}/K}(\alpha_p)$ is totally nonsmooth, so the {\em only} map $X \rightarrow \R_{K^{1/p}/K}(\alpha_p)$ is the $0$ map. Thus it only remains to prove that $\mathscr{V} \in \mathscr{F}$.

So suppose given a $K$-morphism $f\colon X \rightarrow \mathscr{V}$ such that
\[
f|X_1 \times \dots \times \{\infty\}_i \times \dots \times X_n \equiv 0
\]
for all $1 \leq i \leq n$. We must show that $f \equiv 0$. Lemma \ref{resfldclpt} implies that $X_i = \P^1\backslash\{\mu_i^{1/p^{s_i}}\}$ for all $i$, for some $\mu_i \in K$ and $s_i > 0$. Let $L := K(\mu_1^{1/p^{s_1}}, \dots, \mu_n^{1/p^{s_n}})$. Then, by definition, $r := \imp(L/K)$. Lemma \ref{givengensgenerate} implies that, perhaps after renumbering, $L/K$ is generated by $\mu_i^{1/p^{s_i}}$ for $1 \leq i \leq r$. One may assume that $\mu_1 \notin K^p$. (Otherwise replace $\mu_1$ by some $p$-power root and reduce $s_1$. Note that $\mu_1 \notin K^{p^{s_1}}$ because $r > 0$.) Complete $\mu_1 \notin K^p$ to a $p$-basis $\mu_1 = \lambda_1, \lambda_2, \dots, \lambda_m$ of $K$. Then 
\[
\mathscr{V} \simeq \left\{X_0 = \sum_{f \in I} \left(\prod_{i=1}^m\lambda_i^{f(i)}\right)X_f^p\right\} \subset \Ga^I,
\]
where $I := I_1$ is as in Proposition \ref{modspacescomp}. Then $f$ corresponds to a map
\[
g\colon \prod_{i=2}^n X_i \rightarrow \calMor((\P^1\backslash\{\lambda_1^{1/p^{s_1}}\}, \infty), (\mathscr{V}, 0))^+
\]
such that $g$ vanishes whenever any of the components is restricted to $\infty$. We wish to show that $g$, hence also $f$, vanishes. 

Because $\mu_i^{1/p^{s_i}}$, $1 \leq i \leq r$, generate $L/K$, the same holds when we replace $L/K$ by $LE/E$ for any field extension $E$ of $K$ in some ambient field containing $L$. In particular, we apply this with $E = K(\lambda_1^{1/p^{\infty}})$. Note, however, that $\mu_1^{1/p^{s_1}} = \lambda_1^{1/p^{s_1}}$ lies in $E$, so that $LE/E$ is now generated by $< r$ elements, hence $\imp(LE/E) < r$ by Proposition \ref{degofprimminnumofgens}. Let $$V_s := \calMor((\P^1\backslash\{\lambda_1^{1/p^{s_1}}\}, \infty), (\mathscr{V}, 0))^+.$$ Extending scalars to $E$ then yields a map
\[
\pi\circ g_E\colon \prod_{i=2}^nX_i \rightarrow (V_s)_E \xrightarrow{\pi} (V_s)_E/((V_s)_E)_{\Split}
\]
which vanishes whenever any of the components on the left is restricted to $\infty$. Furthermore, if $r' := \imp(\Gamma(D_E, \calO_{D_E})_{\red}/E)$, then $r' < r \leq n-1$. By induction, therefore, since the group on the right is $E$-semiwound, we deduce that $g_E$ lands in $((V_s)_E)_{\Split}$. By Lemma \ref{maxlsplitU'}, there is a totally nonsmooth $K$-subgroup $N \subset V_s$ such that $((V_s)_E)_{\Split} \subset N_E$, so $g$ lands in $N \subset V_s$. Because $\prod_{i=2}^n X_i$ is smooth, and $N$ is totally nonsmooth, it follows that $g = 0$. This completes the induction and the proof of the proposition.
\end{proof}

\section{Rigidity in general}
\label{deductionsection}

In the present section, we deduce the general case of the rigidity theorem \ref{rigiditytheorem} from the special case of almost complete rational curves proved in \S\ref{rigalmostcompleteratlsection}. This will be a consequence of geometric class field theory, so we review here the aspects of this theory that we shall require.

Let $\overline{X}$ be a proper curve over a field $K$, and let $D \subset \overline{X}$ be a finite subscheme. Consider the functor $\{\mbox{$K$-schemes}\} \rightarrow \{\mbox{groups}\}$ which sends a $K$-scheme $T$ to the group of pairs $(\mathscr{L}, \phi)$, where $\mathscr{L} \in \Pic(\overline{X} \times T)$ is a line bundle of relative degree $0$ over $T$, and $\phi\colon \calO_{D_T} \xrightarrow{\sim} \mathscr{L}|_{D_T}$ is a trivialization of $\mathscr{L}$ along $D$. We denote by $\Jac_D(\overline{X})$ the fppf sheafification of this functor. 

\begin{remark}
Sheafification is unnecessary if the map $\Gamma(\overline{X}, \calO_{\overline{X}}) \rightarrow \Gamma(D, \calO_D)$ is injective. See \cite[Rem.\,2.1]{rosmodulispaces}. This injectivity holds, for instance, if $\overline{X}$ is geometrically reduced and geometrically connected (and therefore $\Gamma(\overline{X}, \calO_{\overline{X}}) = K$) and $D \neq \emptyset$.
\end{remark}

The natural map $\Jac_D(\overline{X}) \rightarrow \Jac(\overline{X})$ into the Jacobian of $\overline{X}$ which forgets the trivialization $\phi$ is surjective as a map of fppf sheaves, since any line bundle -- locally on $T$ -- may be trivialized along the finite scheme $D$. Its kernel consists of pairs $(\calO_{\overline{X}_T}, u)$ with $u \in \Gamma(D_T, \calO_T)^{\times}$, up to isomorphism. Two such pairs $u_1$ and $u_2$ are isomorphic precisely when there is an automorphism of $\calO_{\overline{X}_T}$ mapping $u_1$ to $u_2$, or equivalently, when $u_1/u_2$ extends to an element of $\Gamma(\overline{X}_T, \calO_{\overline{X}_T})^{\times}$. Letting $A := \Gamma(\overline{X}, \calO_{\overline{X}})$, therefore, we obtain a canonical exact sequence
\begin{equation}
\label{genjacexactseq}
0 \longrightarrow \R_{D/k}(\Gm)/i(\R_{A/k}(\Gm)) \longrightarrow \Jac_D(\overline{X}) \longrightarrow \Jac(\overline{X}) \longrightarrow 0,
\end{equation}
where the groups on the left are Weil restrictions of scalars, and $i \colon \R_{A/k}(\Gm) \rightarrow \R_{D/k}(\Gm)$ is the  $K$-homomorphism corresponding to the map $A \rightarrow \Gamma(D, \calO_D)$. In particular, $\Jac_D(\overline{X})$ is smooth. (Note that $i$ need not be injective.)

If $x \in \overline{X}^{\rm{sm}}(k)$ is a $K$-point lying in the smooth locus of $\overline{X}$, then associated to $x$ one obtains in the usual manner a $K$-morphism $\overline{X}^{\rm{sm}} \rightarrow \Jac(\overline{X})$ by sending a point $y$ to the line bundle associated to the divisor $[y] - [x]$. Similarly, if $x \notin D$, then one obtains a map $i_x\colon \overline{X}^{\rm{sm}}\backslash D \rightarrow \Jac_D(\overline{X})$ via the map sending $y$ to the pair $(\calO([y]-[x]), \phi)$, where $\phi$ is the canonical trivialization along $D$ of the line bundle associated to the divisor $[y] - [x]$ which is disjoint from $D$. A major result from geometric class field theory says that these maps have an Albanese type property with respect to maps from smooth curves into commutative algebraic groups. We require the following strengthened version for maps into commutative groups not containing $\Ga$.

\begin{theorem}$($$\cite[Th.\,6.7]{rosmodulispaces}$$)$
\label{albanese}
Let $X$ be a smooth curve over a field $K$, with regular compactification $\overline{X}$, and let $x \in X(K)$. Let $D \subset \overline{X}$ be a divisor with support $\overline{X}\backslash X$. Then for any finite type commutative $K$-group scheme $G$ not containing a $K$-subgroup scheme $K$-isomorphic to $\Ga$, the natural $K$-group homomorphism
\[
i_x^*\colon \calHom(\Jac_{D}(\overline{X}), G)^+ \rightarrow \calMor((X, x), (G, 1))^+
\]
is a $K$-group isomorphism.
\end{theorem}

We now use Theorem \ref{albanese} in order to give a reformulation of Proposition \ref{rigalmostcompleteratl} that is interesting in its own right.

\begin{proposition}
\label{rigiditymultiadditive}
Let $K$ be a field of characteristic $p > 0$, $U$ a commutative wound unipotent $K$-group, and for $1 \leq i \leq n$, let $A_i$ be a finite product of finite primitive extension fields of $K$. Also let $A := \prod_{i=1}^n A_i$ and $r := \imp(A/K)$. If $n > r$, then the only multiadditive map $m\colon \prod_{i=1}^n \R_{A_i/K}(\Gm) \rightarrow U$ is the $0$ map.
\end{proposition}

\begin{proof}
Extension of scalars to $K_s$ preserves all hypotheses, thanks to Proposition \ref{degofprimsepblbasechange}, so we may assume that $K$ is separably closed. Since Weil restriction commutes with products of rings, Proposition \ref{degofprimjumpsoverlargerexts} reduces us to the case when each $A_i = L_i$ is a primitive field extension of $K$. We first claim that $m$ vanishes when any of the factors is restricted to $\Gm \subset \R_{L_i/K}(\Gm)$. Indeed, upon such a restriction, we obtain a multiadditive map $m'\colon \Gm \times H \rightarrow U$, where $H := \prod_{j \neq i}\R_{L_i/K}(\Gm)$. Then for any $h \in H(K)$, we obtain a $K$-homomorphism $m'_h\colon \Gm \rightarrow U$ which must vanish, as a unipotent group over a field has no nontrivial cocharacters. Because $H(K)$ is Zariski dense in $H$, it follows that $m' = 0$, as claimed.

The map $m$ therefore descends to a multiadditive map (which we also call $m$ by abuse of notation)
\[
m\colon \prod_{i=1}^n \left(\R_{L_i/K}(\Gm)/\Gm\right) \rightarrow U.
\]
Because $L_i/K$ is primitive, we may write $L_i = K(\alpha_i)$. The exact sequence (\ref{genjacexactseq}) shows that $\R_{L_i/K}(\Gm)/\Gm \simeq \Jac_{[\alpha]}(\P^1_K)$, where $[\alpha]$ denotes the closed point of $\P^1_K$ corresponding to $\alpha$ with its reduced structure. Theorem \ref{albanese} therefore implies that $m$ corresponds to a $K$-morphism
\[
f\colon \prod_{i=1}^n \P^1_K\backslash\{\alpha\} \rightarrow U
\]
which vanishes whenever restricted to $\infty$ in one of the components. By Proposition \ref{rigalmostcompleteratl}, one must have $f = 0$, hence also $m = 0$.
\end{proof}

\begin{lemma}
\label{filttorwdunip}
If $G$ is a solvable $K$-group scheme of finite type not containing a $K$-subgroup isomorphic to $\Ga$, then $G$ admits a filtration all of whose subquotients are either tori, commutative wound unipotent, finite \'etale, or totally nonsmooth.
\end{lemma}

\begin{proof}
The quotient $G/G^{\rm{sm}}$ of $G$ by its maximal smooth subgroup scheme is totally nonsmooth, so we may assume that $G$ is smooth. Since $G/G^0$ is finite \'etale, we may assume that $G$ is also connected. If we let $U$ denote the $K$-unipotent radical of $G$, then $U$ is wound unipotent and $G/U$ is solvable pseudo-reductive, hence commutative \cite[Prop.\,1.2.3]{cgp}. We may therefore assume that $G$ is either wound unipotent or commutative pseudo-reductive. In the latter case, letting $T \subset G$ denote the maximal $K$-torus, $G/T$ is wound unipotent. Indeed, pseudo-reductivity is preserved by separable base change \cite[Prop.\,1.1.9]{cgp}, so one may assume $K = K_s$ for this assertion, and in this case it follows from the non-existence of a nontrivial extension of $\Ga$ by $\Gm$ \cite[Prop.\,2.2.18]{rostateduality}. We have thus reduced to the case in which $G$ is wound unipotent. Because wound unipotent groups may be filtered by commutative wound unipotent groups \cite[Prop.\,B.3.3]{cgp}, we are done.
\end{proof}

\begin{corollary}
\label{GnoGasepext}
If $G$ is a solvable $K$-group scheme of finite type, and $L/K$ is a separable extension field, then $G$ contains a $K$-subgroup scheme isomorphic to $\Ga$ if and only if $G_L$ contains an $L$-subgroup scheme isomorphic to $\Ga$.
\end{corollary}

\begin{proof}
This follows from Lemma \ref{filttorwdunip}, and the fact that all of the types of subquotients in that lemma automatically do not contain subgroups isomorphic to $\Ga$, and are preserved by passage to separable extensions.
\end{proof}

We are now ready to prove the rigidity theorem in general.

\begin{theorem}$($Theorem $\ref{rigiditytheorem}$$)$
\label{rigidity}
Let $K$ be a field of characteristic $p > 0$, let $\overline{X}_1, \dots, \overline{X}_n$ be smooth proper geometrically connected curves over $K$, and let $X_i \subset \overline{X}_i$ be dense open subschemes for $1 \leq i \leq n$ with closed complement $D_i := \overline{X}_i \backslash X_i$. Let $D := D_1 \sqcup \dots \sqcup D_n$, and $r := \imp(\Gamma(D, \calO_D)_{\red}/K)$, the degree of imperfection of $D/K$. For each $1 \leq i \leq n$, let $x_i \in X_i(K)$. Finally, let $G$ be a solvable $K$-group scheme of finite type not containing a $K$-subgroup scheme $K$-isomorphic to $\Ga$. If $n > r$, then the only $K$-morphism $f\colon X_1 \times \dots \times X_n \rightarrow G$ such that, for each $1 \leq i \leq n$, $f|X_1 \times \dots X_{i-1} \times \{x_i\} \times X_{i+1} \times \dots \times X_n = 1_G$ is the constant map to the identity $1_G$.
\end{theorem}

\begin{proof}
The theorem is inherited by extensions: If it is true for $G'$ and $G''$, then it is also true for any extension of $G''$ by $G'$. By Lemma \ref{filttorwdunip}, therefore, and the fact that $\prod X_i$ is smooth and connected, we may assume that $G$ is a torus or commutative wound unipotent. We are free to extend scalars to $K_s$, thanks to Proposition \ref{degofprimsepblbasechange}, so for the rest of the proof assume that $K$ is separably closed. The case when $G$ is a $K$-torus follows immediately from the Rosenlicht Unit Theorem \cite[Cor.\,1.2]{conradrosenlicht}, so assume that $G = U$ is commutative wound unipotent. By Theorem \ref{albanese}, $f$ corresponds to a multiadditive map
\[
m\colon \prod_{i=1}^n \Jac_{D_i}(\overline{X}_i) \rightarrow U,
\]
where by abuse of notation $D_i$ denotes the closed subset $D_i$ with its reduced subscheme structure. We must show that $m = 0$. By (\ref{genjacexactseq}), one has for each $i$ an exact sequence
\[
0 \longrightarrow \R_{D_i/K}(\Gm)/\Gm \longrightarrow \Jac_{D_i}(\overline{X}_i) \longrightarrow \Jac(\overline{X}_i) \longrightarrow 0.
\]
We claim that, for any subset $S \subset \{1, \dots, n\}$, the restriction $m'$ of $m$ to $$\prod_{i \in S} \left(\R_{D_i/K}(\Gm)/\Gm\right) \times \prod_{i\notin S} \Jac_{D_i}(\overline{X}_i)$$ vanishes. The case $S = \emptyset$ will then complete the proof of the theorem.

The proof is by descending induction on $\#S$. First consider the case in which $S = \{1, \dots, n\}$. Lemma \ref{resfldclpt} implies that each $D_i$ is a finite disjoint union of spectra of primitive (finite) field extensions of $K$, so the claim in this case follows from Proposition \ref{rigiditymultiadditive}. Now suppose that $\#S < n$, and that the claim holds for larger subsets of $\{1, \dots, n\}$. For each $j \notin S$, the restriction of $m'$ to $$\prod_{i \in S\cup\{j\}}\left(\R_{D_i/K}(\Gm)/\Gm\right) \times \prod_{i \notin S\cup\{j\}}\Jac_{D_i}(\overline{X}_i)$$ vanishes by induction. It follows that $m'$ descends to a multiadditive map which we by abuse of notation also denote by $m'$
\[
m'\colon\prod_{i \in S} \left(\R_{D_i/K}(\Gm)/\Gm\right) \times \prod_{i\notin S}\Jac_{D_i}(\overline{X}_i) \rightarrow U.
\]
Let $H := \prod_{i \in S} \left(\R_{D_i/K}(\Gm)/\Gm\right)$ and $A := \prod_{i\notin S}\Jac_{D_i}(\overline{X}_i)$. Then for every $h \in H(K)$, we obtain a map $m'_h\colon A \rightarrow U$ whose image contains $0$. (This last is because the product over $i\notin S$ is not the empty scheme because $\#S < n$.) Because $A$ is smooth, geometrically connected, and proper and $U$ is affine, this map must be the constant map to $0 \in U(K)$. Because $H(K)$ is Zariski dense in $H$, it follows that $m' = 0$, which completes the induction and the proof of the theorem.
\end{proof}

\section{Applications to unirational groups}
\label{appunirgpssection}

In this section we will apply the rigidity theorem to the study of unirational algebraic groups. Our first goal is to prove Theorem \ref{rigidityfindegofimp}. We require first some notation and basic lemmas. Let $G$ be a group, and let $(Y_1, y_1), \dots, (Y_m, y_m)$ be pointed sets. Let $g\colon \prod_{i=1}^m Y_i \rightarrow G$ be a (set-theoretic) map. Then we define maps $\Delta^m_{g, y_1, \dots, y_m}\colon \prod_{i=1}^m Y_i \rightarrow G$ recursively as follows. When $m = 1$, we define $\Delta^1_{g,y_1}(z) := g(z)g(y_1)^{-1}$. For $m > 1$, we define $$\Delta^m_{g, y_1, \dots, y_m}(z_1, \dots, z_m) := \Delta^{m-1}_{g_{z_1}, y_2, \dots, y_m}(z_2, \dots, z_m)(\Delta^{m-1}_{g_{y_1}, y_2, \dots, y_m}(z_2, \dots, z_m))^{-1},$$ where for $z \in Y_1$, $g_z\colon \prod_{i=2}^m Y_i \rightarrow G$ denotes the restriction of $g$ to $\prod_{i=2}^m Y_i \simeq \{z\} \times \prod_{i=2}^m Y_i$. A simple induction proves:
\begin{itemize}
\item[(1)] $\Delta^m_{g,y_1,\dots,y_m}(z_1, \dots, z_m) = 1$ whenever $z_i = y_i$ for some $i$.
\item[(2)] If $g(z_1, \dots, z_m) = 1$ whenever $z_i = y_i$ for some $i$, then $\Delta^m_{g,y_1,\dots,y_m} = g$.
\end{itemize}

Now we consider a scheme-theoretic analogue: Let $G$ be a $K$-group scheme, and let $(Y_1, y_1), \dots, (Y_m, y_m)$ be pointed $K$-schemes. Suppose given a morphism of $K$-schemes $g\colon \prod_{i=1}^m Y_i \rightarrow G$. Then for any $K$-scheme $S$, we have the map of sets $\Delta^m_{g,y_1,\dots,y_m}\colon \prod_{i=1}^mY_i(S)\rightarrow G(S)$, and this construction is functorial in $S$. Thus we obtain a map of $K$-schemes $\prod_{i=1}^m Y_i \rightarrow G$ which we also denote by $\Delta^m_{g,y_1,\dots,y_m}$. The assertions (1) and (2) above yield the following lemma.

\begin{lemma}
\label{propsdelta}
Notations as above, we have:
\begin{itemize}
\item[(1)] For any $K$-scheme $S$, and any $z_i \in Y_i(S)$, we have $\Delta^m_{g,y_1,\dots,y_m}(z_1, \dots, z_m) = 1$ whenever $z_i = (y_i)_S$ for some $i$.
\item[(2)] If, for every $K$-scheme $S$, and every $(z_i) \in \prod_i Y_i(S)$, one has $g(z_1, \dots, z_m) = 1$ whenever $z_i = (y_i)_S$ for some $i$, then $\Delta^m_{g,y_1,\dots,y_m} = g$.
\end{itemize}
\end{lemma}

\begin{theorem}$($Theorem $\ref{rigidityfindegofimp}$$)$
\label{rigidityfindegofimpbody}
Let $K$ be a field of degree of imperfection $r$, and let $X_1, \dots, X_n$ be unirational $K$-schemes. For each $X_i$, let $x_i \in X_i(K)$. Finally, let $G$ be a solvable affine $K$-group scheme of finite type not containing a $K$-subgroup scheme $K$-isomorphic to $\Ga$. If $n > r$, then the only $K$-morphism $f\colon X_1 \times \dots \times X_n \rightarrow G$ such that, for each $1 \leq i \leq n$, $f|X_1 \times \dots X_{i-1} \times \{x_i\} \times X_{i+1} \times \dots \times X_n = 1_G$ is the constant map to the identity $1_G$.
\end{theorem}

\begin{proof}
By Corollary \ref{GnoGasepext}, we are free to extend scalars to $K_s$ and thereby assume that $K$ is separably closed. If $K$ is perfect, then $G$ is a torus, so the assertion follows from the Rosenlicht Unit Theorem \cite[Cor.\,1.2]{conradrosenlicht}, so we may assume that $K$ is imperfect -- in particular, infinite. We first prove the following result: If $K$ is a field of degree of imperfection $r$, $n > r$, $(Y_1, y_1), \dots, (Y_n,y_n)$ are pointed $K$-schemes such that $Y_i$ is open in $\A^{d_i}$, then for any $K$-scheme morphism $g\colon \prod_{i=1}^nY_i \rightarrow G$, one has $\Delta^n_{g,y_1,\dots,y_n} = 1$. Indeed, every point of $Y_i$ may be connected to $y_i$ by a curve $C \subset Y_i$ that is isomorphic to an open subscheme of $\A^1$. It therefore suffices to treat the case when each $Y_i$ is open in $\A^1$. Let $\overline{Y}_i \simeq \P^1$ be the regular completion of $Y_i$. The total complementary divisor $D$ of the $Y_i$ satisfies $\imp(D/K) \leq r$ by Proposition \ref{degimpleqdegimperf}. By Lemma \ref{propsdelta}(1), $\Delta^n_{g,y_1,\dots,y_n}$ takes the value $1$ whenever any of the coordinates is $y_i$. It therefore follows from the rigidity theorem \ref{rigidity} that $\Delta^n_{g,y_1,\dots,y_n}$ is identically $1$.

Now we prove the theorem. For each $i$ we have a dominant map $\pi_i\colon U_i \rightarrow X_i$, where $U_i$ is a dense open in $\A^{d_i}$. Pulling the map $f$ back along the induced map $\pi\colon \prod U_i \rightarrow \prod X_i$, we obtain the map $f\circ \pi\colon \prod U_i \rightarrow G$. Let $u_i \in U_i(K)$. By the result of the first paragraph, $\Delta^n_{f\circ\pi,u_1,\dots,u_n} = 1$, hence $\Delta^n_{f,\pi(u_1),\dots,\pi(u_n)} = 1$. Because $K$ is infinite, each $\pi_i$ is dominant, and each $U_i$ is rational, $\pi((\prod U_i)(K))$ is dense in $\prod X_i$, so we conclude that the map $\Delta^n_f\colon (\prod X_i)^2 \rightarrow G$, $$(y_1,\dots,y_n) \times (z_1,\dots,z_n) \mapsto \Delta^n_{f,y_1,\dots,y_n}(z_1,\dots,z_n)$$ vanishes. In particular, $\Delta^n_{f,x_1,\dots,x_n}\colon \prod X_i \rightarrow G$ vanishes. By Lemma \ref{propsdelta}(2) and our assumptions on $f$, $\Delta^n_{f,x_1,\dots,x_n} = f$, so we conclude that $f = 1$.
\end{proof}

We now turn to the proofs of Theorems \ref{unirationalitydescends} and \ref{genbycommunirsubgps}. Crucial to both is the following lemma.

\begin{lemma}
\label{partialfrac}$($Partial Fraction Decomposition$)$
Let $K$ be a field, $X \subset \P^1_K$ a nonempty open subscheme with $\P^1\backslash X = \{x_1, \dots, x_n\}$ with the $x_i$ distinct. Then given $x \in X(K)-\{x_1, \dots, x_n\}$, $U$ a unipotent $K$-group, and $f\colon (X, x) \rightarrow (U, 1)$ a pointed $K$-morphism, $f$ may be written uniquely as a product $f = g_1\dots g_n$ such that each $g_i\colon X \rightarrow U$ extends uniquely to a pointed morphism $(\P^1\backslash\{x_i\}, x) \rightarrow (U, 1)$.
\end{lemma}

We refer to the above result as ``partial fraction decomposition'' because, when $U = \Ga$, it is just the classical statement of the existence of partial fraction decompositions.

\begin{proof}
If $X = \P^1$, then the map is constant and there is nothing to prove, so assume that $X \subsetneq \P^1$. The separatedness of $U$ ensures that the extensions of the $g_i$ are unique, so we do not have to worry about this. One may assume that $x = \infty$. Additionally, because $X$ is smooth and connected, we may assume that $U$ is also smooth and connected. As remarked above, when $U = \Ga$, the lemma is just the classical partial fraction decomposition of a rational function. Thus the lemma holds for $\Ga^r$.

We proceed by induction on ${\rm{dim}}(U)$, the $0$-dimensional case being trivial. So suppose that $U \neq 1$ and place $U$ in an exact sequence of $K$-groups
\begin{equation}
\label{partialfracpfeqn10}
1 \longrightarrow U' \longrightarrow U \longrightarrow U'' \longrightarrow 1
\end{equation}
with $0 \neq U' \subset U$ smooth, connected, central, and $p$-torsion. This may be accomplished as follows: If $r > 0$ is maximal such that $\mathscr{D}_rU \neq 1$, and $\mathscr{D}_rU$ has exponent $p^n$, then take $U' := [p^{n-1}]\mathscr{D}_rU$. Applying \cite[Prop.\,B.1.13]{cgp}, when $K$ is infinite we have an inclusion $j'\colon U' \rightarrow \Ga^r$ with cokernel a subgroup of $\Ga$. When $K$ is perfect, in particular finite, we also have such an inclusion (an isomorphism, in fact) by \cite[Cor.\,B.2.7]{cgp}. Push out the sequence (\ref{partialfracpfeqn10}) by $j'$ to obtain a commutative exact diagram
\[
\begin{tikzcd}
1 \arrow{r} & U' \arrow{r} \arrow[d, hookrightarrow] & U \arrow[d, hookrightarrow, "j"] \arrow{r}{b} & U'' \arrow[d, equals] \arrow{r} & 1 \\
1 \arrow{r} & \Ga^r \arrow{r} \arrow{d} & W \arrow{r}{q} \arrow{d}{\pi} & U'' \arrow{r} & 1 \\
& \Ga \arrow[r, equals] & \Ga &&
\end{tikzcd}
\]
We first construct a partial fraction decomposition for $j\circ f$. By induction, there is a unique decomposition $f_1\dots f_n$ of $q\circ j \circ f\colon X \rightarrow U''$. We abuse notation and also denote by $f_i$ the unique extension of $f_i$ to $\P^1\backslash\{x_i\}$. The obstruction to lifting $f_i$ to a map $\P^1\backslash\{x_i\} \rightarrow W$ lives in ${\rm{H}}^1(\P^1\backslash\{x_i\}, \Ga^r)$, which vanishes because $\P^1\backslash\{x_i\}$ is affine. Thus we may lift each $f_i$ to a map $h_i\colon \P^1\backslash\{x_i\} \rightarrow W$. Then $f' := (j\circ f)(h_1\dots h_n)^{-1}\colon X \rightarrow \Ga^r$, and by the $\Ga^r$ case, we obtain a decomposition for $f'$, which -- thanks to the centrality of $\Ga^r \subset W$ -- yields a decomposition $g_1\dots g_n$ for $j\circ f$. Postcomposing with $\pi$ yields the unique decomposition for $\pi\circ j \circ f = 0$, so we conclude that $\pi\circ g_i = 0$ for all $i$. That is, $g_i$ lands in $U \subset W$. Thus we have found a decomposition for $f$, and it remains to prove uniqueness.

Given two decompositions $f = g_1\dots g_n = s_1\dots s_n$ for $f$, postcomposing with $b$ yields two decompositions for $b\circ f$, so by induction, $b\circ g_i = b\circ s_i$ for all $i$. Then we may write $g_i = t_is_i$ with $t_i\colon \P^1\backslash\{x_i\} \rightarrow U'$. The centrality of $U'$ then implies that $\sum t_i = 0$. Because $U' \subset \Ga^r$, therefore, the $t_i$ yield a decomposition of the $0$ map into $\Ga^r$, hence they are all $0$ by the uniqueness of decompositions for maps into $\Ga^r$.
\end{proof}

We may now prove Theorems \ref{genbycommunirsubgps} and \ref{unirationalitydescends} in the wound unipotent case. We restate both results here for convenience.

\begin{theorem}
\label{unirationalitydescendsunip}
Let $L/K$ be a $($not necessarily algebraic$)$ separable extension of fields, and let $U$ be a wound unipotent $K$-group scheme such that $U_L$ is unirational over $L$. Then $U$ is unirational over $K$.
\end{theorem}

\begin{theorem}$($Theorem $\ref{genbycommunirsubgps}$$)$
\label{genbycommunirsubgpsbody}
If $U$ is a unirational wound unipotent $K$-group scheme, then $U$ is generated by its commutative unirational $K$-subgroups.
\end{theorem}

\begin{proof}[Proof of Theorems $\ref{unirationalitydescendsunip}$ and $\ref{genbycommunirsubgpsbody}$]
We prove Theorem \ref{unirationalitydescendsunip}, and Theorem \ref{genbycommunirsubgpsbody} will come along for the ride. Because $L/K$ is separable, we may write $L$ as a filtered direct limit of smooth $K$-algebras. By standard spreading out techniques, we then obtain a (nonzero) smooth $K$-algebra $R$, an open subscheme $X \subset \P^n_R$, and an $R$-morphism $X \rightarrow U_R$ that is fiberwise dominant. Specializing to an $E$-point of $R$ for some finite Galois extension $E/K$, we thus obtain that $U_E$ is unirational over $E$. We may therefore assume that $L/K$ is finite Galois.

Now we know that $U_L$ is generated by maps from open subschemes of $\P^1_L$ which hit the identity of $U_L$. Applying the partial faction decomposition (Lemma \ref{partialfrac}), we conclude that $U_L$ is generated by $L$-morphisms from open subschemes of $\P^1_L$ whose complements consist of a single closed point. Consider such a map $f\colon X \subset U_L$. We claim that the subgroup of $U_L$ generated by $f$ and its Galois conjugates is commutative. By translation, we may assume that some point $x_0 \in X(L)$ maps to $1 \in U(L)$. By Lemma \ref{resfldclpt} and Proposition \ref{degofprimminnumofgens}, the complementary divisor $D$ of $X$ in $\P^1_L$ has degree of imperfection $\leq 1$ over $K$. Given any Galois conjugate maps $f^{\sigma_i}$ with $\sigma_i \in \Gal(L/K)$, $i = 1, 2$, the induced commutator map $X^{\sigma_1} \times X^{\sigma_2} \rightarrow U_L$ has the property that restriction to $\sigma_i(x_0)$ in the $i$th factor induces the constant map to $1 \in U(L)$. Furthermore, we claim that the degree of imperfection over $K$ of the total complementary divisor $\sigma_1^*(D) \cup \sigma_2^*(D)$ is still $\leq 1$. Indeed, this follows from the very definition of degree of imperfection, as we may on each factor in the coproduct use whichever embedding of the coordinate ring of $\sigma^*(D)$ into $\overline{K}$ that we wish in order to compute $\imp$, and we simply choose $\sigma_i^{-1}$ in the $i$ factor. By the $r = 1$ case of the rigidity theorem \ref{rigidity}, therefore, we conclude that the above commutator map is trivial. Hence $f$ and its Galois conjugates do indeed generate a commutative Galois-invariant $L$-subgroup of $U_L$. This subgroup then descends uniquely to a commutative $K$-subgroup scheme $H$ of $U$, and this $K$-group scheme becomes unirational over $L$. By the already-known descent of unirationality through separable extensions for {\em commutative} groups \cite[Th.\,2.3]{achet}, \cite[Lem.\,2.1]{scavia}, therefore, we deduce that $H$ is unirational over $K$. Because $U_L$ is generated by maps from $X_L$ as above, we conclude that $U$ is generated by commutative unirational $K$-subgroups. Taking $L = K$ in this argument then also implies Theorem \ref{genbycommunirsubgpsbody}.
\end{proof}

In order to complete the proof of Theorem \ref{unirationalitydescends}, we need to pass from the wound unipotent case treated in Theorem \ref{unirationalitydescendsunip} to the general case. The key to accomplishing this will be to relate unirationality of $G$ to unirationality of centralizers of tori in $G$.

\begin{lemma}
\label{GunirimplZGunir}
Let $G$ be a smooth connected affine $K$-group scheme. If $G$ is unirational, and $T \subset G$ is a split $K$-torus, then $Z_G(T)$ is unirational.
\end{lemma}

Later we will prove a stronger result (Proposition \ref{Guniriffcenttori}).

\begin{proof}
We use the ``open cell decomposition.'' If $\beta\colon \Gm \rightarrow T$ is a ``generic'' cocharacter (lying outside the union of finitely many hyperplanes in the cocharacter lattice of $T$), then $Z_G(\beta) = Z_G(T)$. Further, there are closed $K$-subgroups $U^-, U^+ \subset G$ such that the multiplication map
\[
m\colon U^-\times Z_G(\beta) \times U^+ \rightarrow G
\]
is an open immersion \cite[Prop.\,2.1.8(2)(3)]{cgp}. Therefore, $Z_G(T)$ is a direct factor (as a $K$-scheme) of an open subscheme of $G$, hence also unirational.
\end{proof}

We will also require the following result descending unirationality in a special case.

\begin{lemma}
\label{unirdescextuniptori}
Suppose given an exact sequence of smooth connected $K$-group schemes
\[
1 \longrightarrow T \longrightarrow G \xlongrightarrow{\pi} U \longrightarrow 1
\]
with $T$ a torus and $U$ unipotent. If $L/K$ is a separable field extension such that $G$ is unirational over $L$, then $G$ is also unirational over $K$.
\end{lemma}

\begin{proof}
First assume that $U$ is wound. Because $G_L$ is unirational, so is $U_L$, hence $U$ is generated by its commutative $K$-unirational $K$-subgroups by Theorems \ref{unirationalitydescendsunip} and \ref{genbycommunirsubgpsbody}. Let $H \subset U$ be such a $K$-subgroup. Then $\pi^{-1}(H)$ is commutative \cite[Lem.\,4.2]{rospicard}, and we claim that it is unirational over $L_s$, in which case it is also unirational over $K$ \cite[Th.\,2.3]{achet}, \cite[Lem.\,2.1]{scavia}. To prove the claim, we note that $\pi^{-1}(H)_{L_s}$ may be written as an extension
\[
0 \longrightarrow \Gm^n \longrightarrow \pi^{-1}(H)_{L_s} \longrightarrow H_{L_s} \longrightarrow 0
\]
for some $n \geq 0$. We know that $H_{L_s}$ is unirational. Let $X \subset \A^m_{L_s}$ be a dense open subscheme equipped with a dominant $L_s$-morphism $f\colon X \rightarrow H_{L_s}$. Then $X \times_{H_{L_s}} \pi^{-1}(H)_{L_s}$ is a $\Gm^n$-torsor over $X$. Because ${\rm{H}}^1(X, \Gm) = \Pic(X) = 0$, the torsor is trivial, so it is a rational $L_s$-scheme that maps dominantly onto $\pi^{-1}(H_{L_s})$, hence the latter is indeed $L_s$-unirational, as claimed. Therefore, $\pi^{-1}(H)$ is unirational over $K$. Because $G$ is generated by groups of the form $\pi^{-1}(H)$ with $H \subset U$ commutative and unirational, it follows that $G$ is unirational.

Now consider the general case, in which $U$ may fail to be wound. We proceed by induction on $\dim(U)$. If $U$ is wound, then we are done by the previous paragraph. Otherwise, $U$ contains a central $K$-subgroup scheme $U'$ with $U' \simeq \Ga$ \cite[Prop.\,B.3.2]{cgp}. Then $\pi^{-1}(U') \simeq \Ga \times T$ by \cite[Th.\,6.1.1A(ii)]{sga3} and the fact that there are no nontrivial $\Ga$-actions on a torus. Furthermore, $\Ga \subset \pi^{-1}(U')$ is a characteristic subgroup, hence normal in $G$. Then $\overline{G} := G/\Ga$ is unirational over $K$ by induction. Because ${\rm{H}}^1(\overline{G}, \Ga) = 0$, $G = \Ga \times \overline{G}$ as $K$-schemes, whence $G$ is $K$-unirational.
\end{proof}

\begin{proposition}
\label{genbycenttori}
Let $G$ be a smooth connected affine group scheme over a field $K$, and let $T \subset G$ be a $K$-torus. Then $G$ is generated by its $K$-tori together with $Z_G(T)$.
\end{proposition}

\begin{proof}
Let $G_t \subset G$ be the $K$-subgroup generated by the $K$-tori of $G$. Then $G_t \trianglelefteq G$ and $U := G/G_t$ is unipotent \cite[Prop.\,A.2.11]{cgp}. If we let $q\colon G \rightarrow U$ denote the quotient map, then $q(T) = 1$ because $U$ is unipotent. By \cite[Ch.\,IV, 11.14, Cor.\,2]{borelalggroups}, therefore, $q(Z_G(T)) = Z_U(1) = U$. The proposition follows.
\end{proof}

Now we are ready to prove Theorem \ref{unirationalitydescends}.

\begin{theorem}$($Theorem $\ref{unirationalitydescends}$$)$
\label{unirationalitydescendsbody}
Let $L/K$ be a $($not necessarily algebraic$)$ separable extension of fields, and let $G$ be a finite type $K$-group scheme such that $G_L$ is unirational over $L$. Then $G$ is unirational over $K$.
\end{theorem}

\begin{proof}
A unirational $K$-group scheme $H$ is connected. It also has dense set of $K_s$ points. Because the formation of Zariski closures of sets of rational points commutes with arbitrary field extension, it follows that $H$ is geometrically reduced, hence smooth. Finally, we claim that $H$ is affine. This may be checked over an algebraic closure of $K$, hence for the purpose of proving this claim we may assume that $K = \overline{K}$. By Chevalley's Theorem, if $H$ is not affine then it has a nonzero abelian variety quotient, which must also be unirational. But abelian varieties do not admit nonconstant maps from rational $K$-schemes, so this is impossible, hence the claim.

Now we turn to the proof of the theorem. The previous paragraph implies that $G$ is smooth, connected, and affine. We may of course replace $L$ by a separable extension of $L$ and thereby assume that $L$ is separably closed. Let $T \subset G$ be a maximal $K$-torus. Because $G_L$ is unirational over $L$, so is $Z_G(T)_L$ by Lemma \ref{GunirimplZGunir}. The maximality of $T$ ensures that $Z_G(T)/T$ contains no nontrivial torus and is therefore unipotent. Lemma \ref{unirdescextuniptori} then implies that $Z_G(T)$ is unirational over $K$. Because tori are unirational over every field, Proposition \ref{genbycenttori} then implies that $G$ is unirational over $K$ as well.
\end{proof}

For any finite type group scheme over a field $K$, denote by $G_{{\rm{uni}}}$ the maximal unirational $K$-subgroup scheme of $G$.

\begin{corollary}
\label{maxunirinvsepext}
The formation of the maximal unirational subgroup commutes with separable field extension. That is, if $L/K$ is a separable field extension, and $G$ a finite type $K$-group scheme, then $(G_{\rm{uni}})_L = (G_L)_{\rm{uni}}$.
\end{corollary}

\begin{proof}
Clearly $(G_{\rm{uni}})_L \subset (G_L)_{\rm{uni}}$, and we must prove the reverse inclusion. Because $L/K$ is separable, we may write $L$ as a filtered direct limit of smooth $K$-algebras: $L = \varinjlim_i R_i$. Then $(G_L)_{\rm{uni}}$ spreads out to a closed $R_i$-subgroup scheme $H \subset G_{R_i}$ for some $i$. Enlarging $i$ further, there is a fiberwise dominant map to $H$ from an open subscheme of $\P^n_{R_i}$. Specializing to an $E$-point of $R_i$ for some finite Galois extension $E/K$, we obtain a unirational $E$-subgroup scheme $H' \subset G_E$ whose dimension equals that of $(G_L)_{\rm{uni}}$. It follows that $H'_L = (G_L)_{\rm{uni}}$. We may therefore assume that $L/K$ is finite Galois. For any $\sigma \in \Gal(L/K)$, $\sigma^*((G_L)_{\rm{uni}}) \subset G_L$ is still unirational, so $(G_L)_{\rm{uni}}$ is $\Gal(L/K)$-invariant and therefore descends to a $K$-subgroup $G' \subset G$. By Theorem \ref{unirationalitydescendsbody}, $G'$ is unirational over $K$, so $(G_L)_{\rm{uni}} \subset (G_{\rm{uni}})_L$.
\end{proof}

\begin{corollary}
\label{maxunirrnomal}
For any smooth finite type group scheme $G$ over a field $K$, $G_{\rm{uni}}$ is a normal $K$-subgroup scheme of $G$.
\end{corollary}

\begin{proof}
By Corollary \ref{maxunirinvsepext}, we may assume that $K$ is separably closed. For any $g \in G(K)$, $gG_{\rm{uni}}g^{-1} \subset G$ is still unirational, so $G_{\rm{uni}}$ is invariant under $G(K)$-conjugation. But $G$ is smooth and $K$ is separably closed, so $G(K)$ is Zariski dense in $G$, hence $G_{\rm{uni}} \trianglelefteq G$. 
\end{proof}

As a consequence of the proof of Theorem \ref{unirationalitydescendsbody}, we obtain the following result, interesting in its own right.

\begin{proposition}
\label{Guniriffcenttori}
Let $G$ be a smooth connected affine group scheme over a field $K$. Then the following are equivalent:
\begin{itemize}
\item[(i)] $G$ is unirational.
\item[(ii)] $Z_G(T)$ is unirational for every $K$-torus $T \subset G$.
\item[(iii)] $Z_G(T)$ is unirational for some $K$-torus $T \subset G$.
\end{itemize}
\end{proposition}

\begin{proof}
For the implication (i)$\implies$(ii), Theorem \ref{unirationalitydescendsbody} allows us to assume that $K$ is separably closed, in which case this implication is Lemma \ref{GunirimplZGunir}. The implication (ii)$\implies$(iii) is clear, and (iii)$\implies$(i) follows from Proposition \ref{genbycenttori} together with the unirationality of tori.
\end{proof}

\section{Torsors for special groups over rational curves}
\label{torsorssection}

The final goal of the present paper is to apply the rigidity theorem \ref{rigidity} to the study of permawound groups, which were introduced in \cite{rosasuniv}. (Some of the properties of permawound groups were recalled in \S\ref{asunivsection}.) In particular, we will prove that permawound groups are unirational, and that wound permawound groups are commutative.
In this section we prove a technical result (Proposition \ref{torsorsdieratlcurve}) concerning torsors for certain special groups (the ones that arise in the rigidity property of permawound groups; see Theorem \ref{rigidityasuniv}) over rational curves. This will play an important role in our proof that permawound groups are unirational.

We break the proof up into a sequence of lemmas. Fix $r \geq 0$. As in previous sections, we let $I_d$ denote the set of functions $\{1, \dots, r\} \rightarrow \{0, 1, \dots, p^d-1\}$, and $I := I_1$.

\begin{lemma}
\label{beta/(T^{p^n}+lambda)}
Let $K$ be a separably closed field of finite degree of imperfection, and $\vec{\lambda} := [\lambda_1, \dots, \lambda_r]^T$ be a $p$-basis for $K$. Also let $$F_{\vec{\lambda}} := -X_0 + \sum_{f \in I}\left(\prod_{i=1}^r\lambda_i^{f(i)}\right)X_f^p.$$
Then for any $\beta \in K$, $1 \leq i \leq r$, and $n > 0$, one has
\[
\frac{\beta}{T^{p^n} + \lambda_1} \in F_{\vec{\lambda}}(K(T), \dots, K(T)).
\]
\end{lemma}

\begin{proof}
If the assertion is true for $n = 1$, then substituting $T^{p^{n-1}}$ for $T$ proves it for all $n > 0$, so we may assume that $n = 1$. Since $K = K_s$, there exists $\gamma \in K$ such that $\gamma^p-\gamma - \lambda_1^{-1}\beta = 0$. Then
\[
\frac{\beta}{T^{p} + \lambda_1} = F_{\vec{\lambda}}((H_f)_{f \in I}),
\]
where $H_f := 0$ if $f(j) \neq 0$ for some $j \neq 1$, and otherwise
\[
H_f := \begin{cases}
{{p-1}\choose{f(1)-1}}\gamma T^{p-f(1)}/(T^p+\lambda_1), & 0 < f(1) \leq p-1 \\
\lambda_1\gamma/(T^p+\lambda_1), & f(1) = 0.
\end{cases}
\] \hfill \qedhere
\end{proof}

\begin{lemma}
\label{congmodF_lambda}
Assumptions and notation as in Lemma $\ref{beta/(T^{p^n}+lambda)}$. Then for any $G \in K(T)$ and $n > 0$, one has
\begin{itemize}
\item[(i)] $\left(\prod_{i=1}^r\lambda_i^{f(i)}\right)G^{p^n} \in F_{\vec{\lambda}}(K(T), \dots, K(T))$ for $0 \neq f \in I_n$.
\item[(ii)] $G^{p^n} - G \in F_{\vec{\lambda}}(K(T), \dots, K(T))$.
\end{itemize}
\end{lemma}

\begin{proof}
We prove both assertions by induction on $n$. First suppose that $n = 1$. Then for (i), note that $$\left(\prod_{i=1}^r\lambda_i^{f(i)}\right)G^{p} = F_{\vec{\lambda}}((H_g)_g),$$ where $H_g := 0$ unless $g = f$, in which case $H_f:= G$. For (ii), take $H_g:= 0$ unless $g = 0$, in which case $H_g := G$. This proves the base case. For the induction step of (ii), note that the $n = 1$ case with $G$ replaced by $G^{p^{n-1}}$ implies that $G^{p^n} \equiv G^{p^{n-1}} \pmod{F_{\vec{\lambda}}(K(T), \dots, K(T))}$, and this is $\equiv G$ by induction. For the induction step of (i), let $h \in I$ denote the unique element such that $f \equiv h \pmod{p}$, and let $g \in I_{n-1}$ be the unique element with $f = h + pg$. If $h\neq 0$, then the $n = 1$ case of (i) applied with $f$ replaced by $h$ and $G$ replaced by $\left(\prod_{i=1}^r\lambda_i^{g(i)}\right)G^{p^{n-1}}$ implies that
$$\left(\prod_{i=1}^r\lambda_i^{f(i)}\right)G^{p^n} \in F_{\vec{\lambda}}(K(T), \dots, K(T)).$$
If $h = 0$, on the other hand, then we apply the $n = 1$ case of (ii) with $G$ replaced by $\left(\prod_{i=1}^r\lambda_i^{g(i)}\right)G^{p^{n-1}}$ again to conclude that
$$\left(\prod_{i=1}^r\lambda_i^{f(i)}\right)G^{p^n} \equiv \left(\prod_{i=1}^r\lambda_i^{g(i)}\right)G^{p^{n-1}} \pmod{F_{\vec{\lambda}}(K(T), \dots, K(T))}.$$
Because $0 \neq f = pg$, one has $g \neq 0$, hence induction completes the proof.
\end{proof}

\begin{lemma}
\label{betaT^r/(T^p^n+mu)}
Notation and assumptions as in Lemma $\ref{beta/(T^{p^n}+lambda)}$. Then for any $\beta \in K$, $\mu \in K-K^p$, $n > 0$, and $0 \leq r < p^n$, one has
\[
\frac{\beta T^r}{T^{p^n}+\mu} \in F_{\vec{\lambda}}(K(T), \dots, K(T)).
\]
\end{lemma}

\begin{proof}
We first prove the lemma when $\mu = \lambda_1$. For any integer $r$ and $\beta \in K$, Lemma \ref{beta/(T^{p^n}+lambda)} with $\beta$ replaced by $\lambda_1^r\beta^{p^n}$ implies that
\[
\frac{\lambda_1^r\beta^{p^n}(T^{p^n}+\lambda_1)^{p^n-1}}{(T^{p^n}+\lambda_1)^{p^n}} = \frac{\lambda_1^r\beta^{p^n}}{T^{p^n}+\lambda_1} \in F_{\vec{\lambda}}(K(T), \dots, K(T)).
\]
We may uniquely expand the above expression as a sum $\sum_{i=0}^{p^n-1}\lambda_1^iG_i^{p^n}$. By Lemma \ref{congmodF_lambda}, the above expression is then $\equiv G_0 \pmod{F_{\vec{\lambda}}(K(T), \dots, K(T))}$. We now compute $G_0 \in F_{\vec{\lambda}}(K(T), \dots, K(T))$. Let $0 \leq j < p^n$ be the unique integer $\equiv -r \pmod{p^n}$. Then one checks that $G_0$ is
\[
\frac{{{p^n-1}\choose{j}}\beta T^{p^n-1-j}\lambda_1^{(j+r)/p^n}}{T^{p^n}+\lambda_1} \in F_{\vec{\lambda}}(K(T), \dots, K(T)).
\]
Since every digit of $p^n-1$ in its base $p$ expansion is $p-1$, we have $\binom{p^n-1}{j} \not\equiv 0\pmod{p}$ \cite[Lem.\,2.2.7]{rostateduality}. Therefore, since $\beta \in K$ was arbitrary and $p^n-1-j$ ranges over every integer in $[0, p^n)$ as $r$ does, we deduce the lemma in the special case $\mu = \lambda_1$.

To treat the general case, we complete $\mu$ to a $p$-basis $\vec{\mu} := [\mu_1 = \mu, \mu_2, \dots, \mu_r]^T$ of $K$. Then there exist an invertible linear change of coordinates $A \in {\rm{GL}}_r(K)$ and $c \in K^{\times}$ such that $F_{\vec{\mu}}(A\vec{X}) = cF_{\vec{\lambda}}(\vec{X})$ \cite[Lem.\,7.1]{rosasuniv}. We know from the special case of the lemma already treated that any $K$-multiple of $T^r/(T^{p^n}+\mu)$ lies in $F_{\vec{\mu}}(K(T), \dots, K(T))$, and applying a suitable change of coordinates to the input to $F_{\vec{\mu}}$, we obtain the same conclusion for $F_{\vec{\lambda}}$.
\end{proof}

\begin{lemma}
\label{betaT^r/(T^p^n+mu)general}
Notation and assumptions as in Lemma $\ref{beta/(T^{p^n}+lambda)}$. Then for any $\beta \in K$, $\mu \in K$, $n \geq 0$, and $0 \leq r < p^n$, let
\[
G(T) := \frac{\beta T^r}{T^{p^n}+\mu}.
\]
Then there exist $\alpha \in K^{\times}$ and $d \geq 0$ such that $G(\alpha T^{p^d}) \in F_{\vec{\lambda}}(K(T), \dots, K(T))$.
\end{lemma}

\begin{proof}
We proceed by induction on $n$. First consider the base case $n = 0$ (so also $r = 0$). The assertion is trivial if $\beta = 0$, so assume that $\beta \neq 0$. Then we claim that $G(\beta T^p\lambda_1) \in F_{\vec{\lambda}}(K(T), \dots, K(T))$. Indeed, we have
\[
G(\beta T^p/\lambda_1) = \lambda_1/(T^p + \mu\lambda_1/\beta).
\]
If $\mu\lambda_1/\beta \notin K^p$, then the above lies in the image of $F_{\vec{\lambda}}$ by Lemma \ref{betaT^r/(T^p^n+mu)}. Otherwise, write $\mu\lambda_1/\beta = \gamma^p$ with $\gamma \in K$, so that $$G(\beta T^p/\lambda_1) = \lambda_1(1/(T+\gamma))^p \in F_{\vec{\lambda}}(K(T), \dots, K(T))$$ by Lemma \ref{congmodF_lambda}(i). This completes the proof for $n = 0$.

Now suppose that $n > 0$. If $\mu \notin K^p$, then already $G \in F_{\vec{\lambda}}(K(T), \dots, K(T))$ by Lemma \ref{betaT^r/(T^p^n+mu)}. So assume that $\mu = \gamma^p$ with $\gamma \in K$. There are two cases: $p\mid r$ and $p\nmid r$. First consider the case $p\mid r$. Then write $\beta = \sum_{f \in I} \left(\prod_{i=1}^r\lambda_i^{f(i)}\right)c_f^p$ with $c_f \in K$. Then we have
\[
G(T) = \sum_{f \in I}\left(\prod_{i=1}^r\lambda_i^{f(i)}\right)\left(\frac{c_fT^{r/p}}{T^{p^{n-1}}+\gamma}\right)^p \equiv \frac{c_0T^{r/p}}{T^{p^{n-1}}+\gamma} \pmod{F_{\vec{\lambda}}(K(T), \dots, K(T))},
\]
the last equality by Lemma \ref{congmodF_lambda}, and we are done by induction. Now suppose that $p\nmid r$. We may assume that $\beta \neq 0$. Because $K$ is separably closed, we may choose $\alpha \in K$ such that $\alpha^{r-p^n} = \lambda_1/\beta$. Then
\[
G(\alpha T^p) = \lambda_1\left(\frac{T^r}{T^{p^n}+\gamma \alpha^{-p^{n-1}}} \right)^p \in F_{\vec{\lambda}}(K(T), \dots, K(T)),
\]
the last containment by Lemma \ref{congmodF_lambda}(i).
\end{proof}

\begin{lemma}
\label{monomsinFlambda}
Notation and assumptions as in Lemma $\ref{beta/(T^{p^n}+lambda)}$. Then for any monomial $G(T) = cT^n \in K[T]$, there exist $\alpha \in K^{\times}, d \geq 0$ such that $G(\alpha T^{p^d}) \in F_{\vec{\lambda}}(K(T), \dots, K(T))$.
\end{lemma}

\begin{proof}
We proceed by induction on $n$. When $n = 0$, the fact that $K$ is separably closed furnishes $\beta \in K$ such that $\beta^p-\beta = c$, so we are done by Lemma \ref{congmodF_lambda}(ii). Now suppose that $n > 0$. We split the proof into two cases: $p\mid n$ and $p\nmid n$. First suppose that $p\mid n$. Then we write $c = \sum_{f \in I}\left(\prod_{i=1}^r\lambda_i^{f(i)}\right)c_f^p$ with $c_f \in K$. We have
\[
G(T) = \sum_{f \in I}\left(\prod_{i=1}^r\lambda_i^{f(i)}\right)(c_fT^{n/p})^p \equiv c_0T^{n/p} \pmod{F_{\vec{\lambda}}(K(T), \dots, K(T))},
\]
the last equality by Lemma \ref{congmodF_lambda}, and we are done by induction. Now suppose, on the other hand, that $p\nmid n$. We may assume that $c \neq 0$. Because $K$ is separably closed, there exists $\alpha \in K$ such that $\alpha^n = c^{-1}\lambda_1$. Then $G(\alpha T^p) = \lambda_1(T^n)^p$, and this lies in the image of $F_{\vec{\lambda}}$ by Lemma \ref{congmodF_lambda}(i).
\end{proof}

We may now prove the key technical result of this section. The group $\mathscr{V}$ in the proposition below is the one defined by (\ref{Vdef}).

\begin{proposition}
\label{torsorsdieratlcurve}
Let $K$ be a separably closed field of finite degree of imperfection, and let $V$ denote either $\mathscr{V}$ or $\R_{K^{1/p}/K}(\alpha_p)$. Then for any dense open subscheme $X \subset \P^1_K$, and any $V$-torsor $\mathscr{T}$ over $X$, there is a dense open subscheme $Y \subset \P^1_K$ and dominant $K$-morphism $f\colon Y \rightarrow X$ such that $f^*(\mathscr{T})$ is trivial.
\end{proposition}

\begin{proof}
The assertion is a rational one: We may replace $X, Y$ by $K(T)$. We first treat the case $V = \mathscr{V}$. We may functorially identify ${\rm{H}}^1(K(T), \mathscr{V})$ with the quotient $$K(T)/F_{\vec{\lambda}}(K(T), \dots, K(T))$$ in the notation of Lemma \ref{beta/(T^{p^n}+lambda)}. Thus the proposition for $V = \mathscr{V}$ may be rephrased as the assertion that, for any $G(T) \in K(T)$, there is a $K$-endomorphism $\phi$ of $K(T)$ such that $\phi(G(T)) \in F_{\vec{\lambda}}(K(T), \dots, K(T))$. In fact, we will show that one may take $\phi$ to be of the form $T \mapsto \alpha T^{p^d}$ for suitable $\alpha \in K^{\times}, d \geq 0$. Because $K$ is separably closed, every monic polynomial in $K[T]$ may be written as a product of polynomials of the form $T^{p^n}+\mu$ with $\mu \in K$. Thus we may write every element of $K(T)$ as a sum of terms lying in $S$, where $S$ denotes the set consisting of monomials and of rational functions of the form $\beta T^r/(T^{p^n}+\mu)$ with $\beta, \mu \in K$ and $0 \leq r < p^n$. Note that the set $S$ is mapped into itself under the $K$-endomorphism $T \mapsto \alpha T^{p^d}$ of $K(T)$, as is the set $F_{\vec{\lambda}}(K(T), \dots, K(T))$. Write $G(T)$ as a sum of elements of $S$. We will show that there exist suitable $\alpha, d$ for $G$ by induction on the number $m$ of terms in this sum. The base case $m = 0$ is trivial, so assume that $m > 0$. So we may write $G(T) = s + H(T)$, where $s \in S$ and $H$ is a sum of $m-1$ terms of $S$. Lemmas \ref{betaT^r/(T^p^n+mu)general} and \ref{monomsinFlambda} imply that there exist $\alpha, d$ such that $s$ is killed (modulo $F_{\vec{\lambda}}(K(T), \dots, K(T))$) by $T \mapsto \alpha T^{p^d}$. Because $S$ is mapped into itself by this map, it follows by induction that $G$ may be killed by a map of the form $T \mapsto \alpha'T^{p^{d'}}$. This completes the proof of the proposition for $V = \mathscr{V}$.

The case $V = \R_{K^{1/p}/K}(\alpha_p)$ is easier. By \cite[Prop.\,7.4]{rosasuniv}, this group is described by the equation
\[
F(T) := \sum_{f \in I}\left(\prod_{i=1}^r\lambda_i^{f(i)}\right)X_f^p = 0.
\]
Thus, as above, it is enough to construct, for any $G(T) \in K(T)$, a $K$-endomorphism $\phi$ of $K(T)$ such that $\phi(G) \in H(K(T), \dots, K(T))$. In fact, because the $\lambda_i$ form a $p$-basis for $K$, one may take $\phi(T) := T^p$.
\end{proof}

\section{The groups $\mathscr{V}_{n,\vec{\lambda}}$}
\label{Vnlambdasection}

In this section we study certain important permawound unipotent groups that will play a fundamental role in our applications. The groups in question are described in the following definition.

\begin{definition}
\label{Vndef}
Let $K$ be a field of finite degree of imperfection $r$, and let $\vec{\lambda} := [\lambda_1, \dots, \lambda_r]$ be a $p$-basis. For $n \geq 0$, let $I_n$ denote the set of functions $\{1, \dots, r\} \rightarrow \{0, \dots, p^n-1\}$. For $n > 0$, define a $p$-polynomial $F_{n, \vec{\lambda}} \in K[X, X_f \mid f \in I_n, f \not\equiv 0\pmod{p}]$ by the formula
\begin{equation}
\label{Vndefeqn}
F_{n, \vec{\lambda}} = -X + X^p + \sum_{\substack{f \in I_n\\f\not\equiv0\pmod{p}}} \left(\prod_{i=1}^r\lambda_i^{f(i)}\right) X_f^{p^n}.
\end{equation}
Then we define the $K$-group $\mathscr{V}_{n, \vec{\lambda}}$ to be the vanishing locus of $F_{n, \vec{\lambda}}$. When $K$ has degree of imperfection $1$, if $\vec{\lambda} =[\lambda]$, then we also denote $\mathscr{V}_{n,\vec{\lambda}}$ by $V_{n,\lambda}$.
\end{definition}

Note that $\mathscr{V}_{1, \vec{\lambda}} = \mathscr{V}$ when $K$ is separably closed. By replacing $X$ in the expression $F_{n,\vec{\lambda}}$ with the universal reduced $p$-polynomial $$\sum_{g \in I_{n-1}} \left(\prod_{i=1}^r \lambda_i^{g(i)}\right)Y_g^{p^{n-1}},$$ one sees that the principal part of $F_{n, \vec{\lambda}}$ is reduced and universal, hence $V_{n, \vec{\lambda}}$ is wound and permawound by \cite[Prop.\,2.4, Th.\,6.10]{rosasuniv}.

For every $n > 0$, we define a surjection $\phi_n\colon \mathscr{V}_{n+1,\vec{\lambda}} \twoheadrightarrow \mathscr{V}_{n,\vec{\lambda}}$ via the map which sends $(X, (X_f)_{f \in I_{n+1}, f\not\equiv0\pmod{p}})$ to $(X, (Z_g)_{g \in I_{n}, g\not\equiv0\pmod{p}})$, where
\begin{equation}
\label{quotofVnpowpfeqn1}
Z_g := \sum_{\substack{f\in I_{n+1}\\f\equiv g \pmod{p^n}}} \left( \prod_{i=1}^r\lambda_i^{(f(i)-g(i))/p^n} \right)X_f^p.
\end{equation}
Note in particular that this yields an exact sequence
\begin{equation}
\label{quotofVnpowpfeqn2}
0 \longrightarrow (\R_{K^{1/p}/K}(\alpha_p))^{p^{nr}-p^{(n-1)r}} \longrightarrow \mathscr{V}_{n+1,\vec{\lambda}} \xlongrightarrow{\phi_n} \mathscr{V}_{n,\vec{\lambda}} \longrightarrow 0.
\end{equation}

One has the following lemma.

\begin{lemma}
\label{charpullbackphi}
For any $\chi \in \Hom(\mathscr{V}_{n,\vec{\lambda}}, \Ga)$, $\phi_n^*(\chi)$ is a $K$-linear combination of $p$th powers of elements of $\Hom(\mathscr{V}_{n+1,\vec{\lambda}}, \Ga)$.
\end{lemma}

\begin{proof}
Every element of $\Hom(\mathscr{V}_{n,\vec{\lambda}}, \Ga)$ is a $p$-polynomial in $X$ and the $Z_g$ by \cite[Prop.\,6.4]{rosmodulispaces}, so we only need to show that $X$ and $Z_g$ pull back along $\phi_n$ in the desired manner. For $X$, there is not even a need to pull back, as the equation for $\mathscr{V}_{n,\vec{\lambda}}$ shows that $X$ is already a $K$-linear combination of $p$th powers of elements of $\Hom(\mathscr{V}_{n,\vec{\lambda}}, \Ga)$. For $Z_g$, the assertion follows from the equation (\ref{quotofVnpowpfeqn1}) defining $\phi_n$. This proves the lemma.
\end{proof}

The significance of the groups $\mathscr{V}_{n,\vec{\lambda}}$ to the study of permawound groups arises from the following proposition.

\begin{proposition}
\label{quotofVnpow}
Let $K$ be a separably closed field of finite degree of imperfection $r$, and let $\vec{\lambda} := [\lambda_1, \dots, \lambda_r]$ be a $p$-basis for $K$. Let $U$ be a commutative, $p$-torsion, permawound $K$-group scheme. Then for some $n, m$, there is a surjective $K$-homomorphism $(\mathscr{V}_{n, \vec{\lambda}})^m \twoheadrightarrow U$.
\end{proposition}

\begin{proof}
Note first that $\mathscr{V} = \mathscr{V}_{1,\vec{\lambda}}$ admits a $\Ga^{p^r-1}$ quotient; for instance, take the map $\mathscr{V} \rightarrow \Ga^{p^r-1}$ sending $(X, X_f)$ to $(X_f)$. It follows that the proposition holds for $U$ a  vector group. In general, $U$ is the product of a vector group and a wound group that is necessarily permawound \cite[Props.\,2.9, 5.4]{rosasuniv}. We are therefore reduced to the case in which $U$ is wound.

We proceed by induction on ${\rm{dim}}(U)$. When $U = 1$ the assertion is trivial, so assume that $U$ is nontrivial. By Theorem \ref{rigidityasuniv}, $U$ contains a $K$-subgroup scheme $W$ isomorphic to either $\R_{K^{1/p}/K}(\alpha_p)$ or $\mathscr{V}$ and with wound quotient $\overline{U} := U/W$. By induction, there is a surjection $(\mathscr{V}_{n, \vec{\lambda}})^m \twoheadrightarrow \overline{U}$. Let $U'$ denote the $K$-group $U \times_{\overline{U}} (\mathscr{V}_{n,\vec{\lambda}})^m$ that naturally surjects onto $U$; call this surjection $\pi$. Then $U'$ sits in an exact sequence
\[
0 \longrightarrow W \longrightarrow U' \longrightarrow (\mathscr{V}_{n,\vec{\lambda}})^m \longrightarrow 0.
\]
If $W \simeq \mathscr{V}$, then $U' \simeq \mathscr{V} \times (\mathscr{V}_{n, \vec{\lambda}})^m$ \cite[Cor.\,8.3]{rosasuniv}, so we are done in this case using the maps $\phi_n$ defined above to obtain a surjection $\mathscr{V}_{n, \vec{\lambda}} \twoheadrightarrow \mathscr{V}_{1,\vec{\lambda}} = \mathscr{V}$. So assume for the rest of the proof that $W \simeq \R_{K^{1/p}/K}(\alpha_p)$.

Let $$G = \sum_{f \in I_1} \left(\prod_{i=1}^r \lambda_i^{f(i)}\right)Y_f^p,$$ so that $W$ sits in the exact sequence
\[
0 \longrightarrow W \xlongrightarrow{i} \Ga^{I_1} \xlongrightarrow{G} \Ga \longrightarrow 0.
\]
Then $U' \in \Ext^1((\mathscr{V}_{n,\vec{\lambda}})^m, W)$, and the extension $i_*(U') \in \Ext^1((\mathscr{V}_{n,\vec{\lambda}})^m, \Ga^{I_1})$ is a $K$-group scheme with trivial Verschiebung, because $U'$ is. It follows that $U'$ is of the form $\delta(\chi)$ for some $\chi \in \Hom((\mathscr{V}_{n,\vec{\lambda}})^m, \Ga)$, where $\delta\colon \Hom((\mathscr{V}_{n,\vec{\lambda}})^m, \Ga) \rightarrow \Ext^1((\mathscr{V}_{n,\vec{\lambda}})^m, W)$ is the connecting map \cite[Prop.\,2.8]{rosasuniv}. By Lemma \ref{charpullbackphi}, $(\phi_n^m)^*(\chi) \in \Hom((\mathscr{V}_{n+1,\vec{\lambda}})^m, \Ga)$ is a $K$-linear combination of $p$th powers of elements of $\Hom((\mathscr{V}_{n+1,\vec{\lambda}})^m, \Ga)$. It follows that $(\phi_n^m)^*(\chi) \in G_*(\Hom((\mathscr{V}_{n+1,\vec{\lambda}})^m, \Ga^{I_1}))$. Thus $(\phi_n^m)^*(U') \in \Ext^1((\mathscr{V}_{n+1,\vec{\lambda}})^m, W)$ is split. That is, we have a map $f\colon (\mathscr{V}_{n+1,\vec{\lambda}})^m \rightarrow U'$ whose composition with $U' \rightarrow (\mathscr{V}_{n,\vec{\lambda}})^m$ is $\phi_n^m$ and in particular surjective. We claim that $\pi\circ f\colon (\mathscr{V}_{n+1,\vec{\lambda}})^m \rightarrow U$ is surjective. To see this, let $V := \im(\pi\circ f) \subset U$. Because $\im(f)$ surjects onto the quotient $(\mathscr{V}_{n,\vec{\lambda}})^m$ of $U'$, it follows that $U/V$ is a quotient of $W \simeq \R_{K^{1/p}/K}(\alpha_p)$. It is also a quotient of $U$, hence smooth. It follows from \cite[Prop.\,7.7]{rosasuniv} that $U/V$ is a vector group. On the other hand, $V$ is a quotient of the permawound $K$-group $(\mathscr{V}_{n+1,\vec{\lambda}})^m$. It follows that an extension of it by a nontrivial vector group cannot be wound, so the vector group $U/V$ must be trivial. That is, $\pi \circ f$ is surjective.
\end{proof}

The following important result gives an alternative description of $\mathscr{V}$ when $K$ has degree of imperfection $1$.

\begin{proposition}$($$\cite[Ch.\,VI, Prop.\,5.3]{oesterle}$$)$
\label{eqnforweilrestgm}
Let $K$ be a field of degree of imperfection $1$, and let $\lambda \in K-K^p$. Then $\R_{K^{1/p}/K}(\Gm)/\Gm$ is isomorphic to the $K$-group scheme
\[
\left\{X_{p-1} = \sum_{j=0}^{p-1}\lambda^jX_j^p\right\} \subset \Ga^p.
\]
\end{proposition}

\begin{proposition}
\label{weilrestasuniv}
Let $K$ be a field of degree of imperfection $1$. Then for all $n \geq 0$, $\R_{K^{1/p^n}/K}(\Gm)/\Gm$ is permawound.
\end{proposition}

\begin{proof}
For $n = 0$ the assertion is immediate. We prove the result for general $n$ by induction. For $n \geq 1$, we have an exact sequence
\[
0 \longrightarrow \frac{\R_{K^{1/p^n}/K}(\mu_p)}{\mu_p} \longrightarrow \frac{\R_{K^{1/p^n}/K}(\Gm)}{\Gm} \xlongrightarrow{[p]} \frac{\R_{K^{1/p^{n-1}}/K}(\Gm)}{\Gm} \longrightarrow 0.
\]
In particular, the case $n=1$ implies that $\R_{K^{1/p^n}/K}(\mu_p)/\mu_p \simeq \R_{K^{1/p^n}/K}(\Gm)/\Gm$. The latter group is permawound by Proposition \ref{eqnforweilrestgm} and \cite[Th.\,6.10]{rosasuniv}. The proof in general now follows by induction and \cite[Props.\,5.6, 6.9]{rosasuniv}.
\end{proof}

\begin{proposition}
\label{unirimpasuniv}
If $K$ is a field of degree of imperfection $1$, then every unirational wound unipotent $K$-group scheme is permawound.
\end{proposition}

\begin{proof}
From the definition, one sees that permawoundness over $K_s$ implies the same over $K$, so we may assume that $K$ is separably closed. If $U_{\rm{split}} \subset U$ denotes the maximal split unipotent $K$-subgroup, then $U$ is permawound if and only if $U/U_{\rm{split}}$ is, as follows from the definition of permawoundness. We may therefore assume that $U$ is wound, in which case it is necessarily commutative by Theorem \ref{nilpotclass}. Because $U$ is unirational, it is generated by pointed maps $(X, \infty) \rightarrow (U, 0)$ with $X$ an open subscheme of $\P^1_K$ containing $\infty$. The generalized Jacobian of $X$ (see the discussion in \S\ref{deductionsection} and in particular the exact sequence (\ref{genjacexactseq})) is then a quotient of $\R_{A/K}(\Gm)$ for some finite $K$-algebra $A$. By \cite[\S10.3,Th.\,2]{neronmodels}, it follows that $U$ is generated by homomorphisms from such $K$-groups. If $A_{\red}$ is the reduced quotient of $A$, then the natural map $\R_{A/K}(\Gm) \rightarrow \R_{A_{\red}/K}(\Gm)$ has split unipotent kernel. Therefore, $U$ is generated by maps from such groups with $A$ reduced. Writing $A$ as a product of finite field extensions of $K$, therefore, and using the fact that $K$ is separably closed with degree of imperfection $1$, $U$ is generated by maps from $\R_{K^{1/p^n}/K}(\Gm)$ for varying $n$. But $U$ is unipotent, so in fact it is generated by maps from groups of the form $\R_{K^{1/p^n}/K}(\Gm)/\Gm$. Thus $U$ is a quotient of some finite product of such groups. Because these groups are permawound by Proposition \ref{weilrestasuniv}, so is $U$.
\end{proof}

Proposition \ref{eqnforweilrestgm} shows that $\mathscr{V}$ is unirational when $K$ has degree of imperfection $1$. We will also require the following result which implies the same thing for $\mathscr{V}_{n,\vec{\lambda}}$.

\begin{proposition}
\label{V_nunirdegimp1}
Let $K$ be a field of degree of imperfection $1$. Then any commutative permawound unipotent $K$-group scheme is unirational.
\end{proposition}

\begin{proof}
By Theorem \ref{unirationalitydescendsbody} (or even the commutative case which was known before) and \cite[Props.\,6.7, 6.9]{rosasuniv}, we may assume that $K$ is separably closed. Let $U$ be a commutative permawound $K$-group scheme. If $U_{\rm{split}} \subset U$ is the maximal split unipotent $K$-subgroup, and $\overline{U} := U/U_{\rm{split}}$, then left-multiplication makes $U$ into a $U_{\rm{split}}$-torsor over $\overline{U}$. Because $\overline{U}$ is affine and $U_{\rm{split}}$ is split, ${\rm{H}}^1(\overline{U}, U_{\rm{split}}) = 0$, so that $U = U_{\rm{split}} \times \overline{U}$ as $K$-schemes. Because $U$ is permawound, so is $\overline{U}$, so we are reduced to the case in which $U$ is semiwound. Then \cite[Prop.\,6.2]{rosasuniv} shows that $U$ must be wound.

Let $G \subset U$ be the maximal unirational $K$-subgroup scheme. We aim to show that $G = U$. We proceed by induction on ${\rm{dim}}(U)$, the $0$-dimensional case being a triviality. So assume that $U \neq 0$. By Theorem \ref{rigidityasuniv}, there is an exact sequence
\[
0 \longrightarrow W \longrightarrow U \longrightarrow U' \longrightarrow 0
\]
with $W \simeq \mathscr{V}$ or $\R_{K^{1/p}/K}(\alpha_p)$ and $U'$ wound (and permawound). Then $U'$ is unirational by induction. Proposition \ref{torsorsdieratlcurve} then implies that there are morphisms $X_i \rightarrow U$ from open subschemes of $\P^1_K$ whose images generate a subgroup which surjects onto the quotient $U'$ of $U$. It follows that $G \twoheadrightarrow U'$. If $W \simeq \mathscr{V}$, then $W$ is unirational by Proposition \ref{eqnforweilrestgm}, so $G = U$. On the other hand, suppose that $W \simeq \R_{K^{1/p}/K}(\alpha_p)$. Then $U/G$ is a smooth quotient of $W$, therefore a vector group by \cite[Prop.\,7.7]{rosasuniv}. Thus we have for some $d \geq 0$ an exact sequence
\[
G \longrightarrow U \longrightarrow \Ga^d \longrightarrow 0.
\]
The group $G$ is permawound by Proposition \ref{unirimpasuniv}, so if $d > 0$, then it follows that $U$ contains a copy of $\Ga$, a contradiction, so we must have $d = 0$. That is, $G = U$. This completes the induction and the proof of the proposition.
\end{proof}

\begin{lemma}
\label{largergpVnlambdadegimp1}
Let $K$ be a field of degree of imperfection $1$. Then $\mathscr{V}_{n,\lambda}$, with equation given by the vanishing of $F_{n,\vec{\lambda}}$ $($see $($\ref{Vndefeqn}$)$$)$ is a subgroup of the group
\[
X = X^{p^{n-1}} + \sum_{j=0}^{n-2}\sum_{\substack{0 < \ell < p^n\\p\nmid \ell}}\lambda^{\ell p^j}X_{\ell}^{p^{n+j}}.
\]
\end{lemma}

\begin{proof}
We prove the following more general result by induction on $m$. For all $m \geq 0$, $\mathscr{V}_{n,\lambda}$ is a subgroup of
\[
X = X^{p^m} + \sum_{j=0}^{m-1}\sum_{\substack{0 < \ell < p^n\\p\nmid \ell}}\lambda^{\ell p^j}X_{\ell}^{p^{n+j}}.
\]
The case $m = n-1$ yields the lemma. For $m = 0$ the assertion is immediate. Now suppose that $m > 0$ and that the lemma holds for $m-1$. Then we have for $(X, (X_\ell)_{\ell}) \in \mathscr{V}_{n,\lambda}$,
\begin{align*}
X & = X^{p^{m-1}} + \sum_{j=0}^{m-2}\sum_{\substack{0 < \ell < p^n\\p\nmid \ell}}\lambda^{\ell p^j}X_{\ell}^{p^{n+j}} \\
&= \left(X^p + \sum_{\substack{0 < \ell < p^n\\p\nmid \ell}}\lambda^\ell X_{\ell}^{p^n}\right)^{p^{m-1}} + \sum_{j=0}^{m-2}\sum_{\substack{0 < \ell < p^n\\p\nmid \ell}}\lambda^{\ell p^j}X_{\ell}^{p^{n+j}} \hspace{.2 in} \mbox{by (\ref{Vndefeqn})} \\
&= X^{p^m} + \sum_{j=0}^{m-1}\sum_{\substack{0 < \ell < p^n\\p\nmid \ell}}\lambda^{\ell p^j}X_{\ell}^{p^{n+j}}. \qedhere
\end{align*}
\end{proof}

Let $\lambda_1, \dots, \lambda_r$ be a $p$-basis for $K$. We seek to leverage our understanding of the groups $\mathscr{V}_{n,\vec{\lambda}}$ in the degree of imperfection $1$ case to say something about the general finite degree of imperfection case. The key is the construction of a suitable multi-additive pairing as in the following proposition.

\begin{proposition}
\label{multiadd}
Let $\vec{\lambda} := [\lambda_1, \dots, \lambda_r]^T$ be a $p$-basis for the imperfect field $K$, and let $\mathscr{V}_{n,\lambda_i}$ denote the group over $\F_p(\lambda_i)$ introduced in Definition $\ref{Vndef}$. There is a multi-additive pairing
\[
b\colon \prod_{i=1}^r (\mathscr{V}_{n,\lambda_i})_K \rightarrow \mathscr{V}_{n,\vec{\lambda}}
\]
such that the image of $b$ generates the group $\mathscr{V}_{n,\vec{\lambda}}$.
\end{proposition}

\begin{proof}
For ease of notation, we write $\mathscr{V}_{n,\lambda_i}$ as the group
\[
X_i = X_i^p + \sum_{\substack{0 < j < p^n\\p\nmid j}}\lambda_i^jX_{i,j}^{p^n}
\]
and $\mathscr{V}_{n,\vec{\lambda}}$ as
\[
Z = Z^p + \sum_{\substack{f \in I_n\\p\nmid f}} \left(\prod_{i=1}^r\lambda_i^{f(i)}\right)Z_f^{p^n}.
\]
(Same equations, but we have renamed the variables.) For an integer $s > 0$, let $v_p(s)$ denote the maximal power of $p$ dividing $s$. Then we define $b(\prod_{i=1}^r (X_i, (X_{i,j})_j)) := (Z, (Z_f)_f)$, where
\[
Z := \prod_{i=1}^rX_i
\]
\begin{equation}
\label{Zfdefeqn}
Z_f := \left(\prod_{f(i)=0}X_i\right)\left(\prod_{f(i) \neq 0}\sum_{\substack{0 < \ell < p^n\\ \ell \equiv f(i)/p^{v_p(f(i))}
\pmod{p^{n-v_p(f(i))}}}}\lambda_i^{(\ell - f(i)/p^{v_p(f(i))})/p^{n-v_p(f(i))}}X_{i,\ell}^{p^{v_p(f(i))}}  \right).
\end{equation}
This is a multi-additive map. Next we check that it lands in $\mathscr{V}_{n,\vec{\lambda}}$. We compute
\begin{align*}
Z - Z^p &= \prod_{i=1}^r X_i - \prod_{i=1}^r X_i^p \\
&= \prod_{i=1}^r\left(X_i^p + \sum_{\substack{0 < j < p^n\\p\nmid j}}\lambda_i^jX_{i,j}^{p^n}\right) - \prod_{i=1}^rX_i^p\\
&= \sum_{\emptyset \neq S \subset \{1, \dots, r\}} \left(\prod_{i \in S} \sum_{\substack{0 < j < p^n\\p\nmid j}}\lambda_i^jX_{i,j}^{p^n} \right)\left(\prod_{i \notin S}X_i^p\right) \hspace{.2 in} \mbox{by expanding the product}
\end{align*}
\begin{align*}
&= \sum_{\emptyset \neq S \subset \{1, \dots, r\}} \left(\prod_{i \in S} \sum_{\substack{0 < j < p^n\\p\nmid j}}\lambda_i^jX_{i,j}^{p^n} \right) \left( \prod_{i\notin S} \left[X_i^{p^n} + \sum_{j=0}^{n-2}\sum_{\substack{0 < \ell < p^n\\p\nmid \ell}} \lambda_i^{\ell p^{j+1}}X_{i,\ell}^{p^{n+j+1}}\right] \right) \hspace{.1 in} \mbox{by Lemma \ref{largergpVnlambdadegimp1}}.
\end{align*}
Because the sum is over {\em nonempty} $S$, we see that expanding out the above product yields a sum of terms of the form $$\prod_{i=1}^r\lambda_i^{f(i)}G_i((X_{i,j})_{j})^{p^n}$$ with $G_i$ a $p$-polynomial in $(X_{i,j})_j$ and $f \in I_n$ such that $p\nmid f$. Thus we obtain a pairing into $\mathscr{V}_{n,\vec{\lambda}}$, and in fact, collecting the terms corresponding to a given $f$ exactly yields the formula for $Z_f$ given above. Thus the pairing lands inside $\mathscr{V}_{n,\vec{\lambda}}$, and it only remains to verify that its image generates this group.

We claim that there is no nonzero $p$-polynomial $H \in K[Y_f\mid f \in I_n, p\nmid f]$ such that $H((Z_f)_f)$ vanishes when pulled back along $b$. To prove this, let us first note that, given two polynomials $F, G \in K[T_i, T_{i,j}]_{i,j}$ whose degrees in each $T_i$ are $< p$, then $F(X_i, X_{i,j}) = G(X_i, X_{i,j})$ only if $F = G$ as polynomials \cite[Lem.\,6.3]{rosmodulispaces}. Now suppose that $H$ is a $p$-polynomial as before which vanishes when pulled back along $b$. We wish to show that $H = 0$. First suppose that some $Z_f$ with the property that $f(i) = 0$ for some $i$ appears in $H$. Let $M$ denote the sum of those terms of $H$ of minimal degree $p^d$ involving some such $Z_f$, and let $Q := H - M$. Write
\[
M = \sum_{\substack{f\\ \mbox{some } f(i)= 0}} c_fZ_f^{p^d}.
\]
An easy induction using the equation (\ref{Vndefeqn}) for $\mathscr{V}_{n,\lambda_i}$ shows that $X_i^{p^m} - X_i$ may be expressed as a $p$-polynomial in the $X_{i,j}$ for all $m \geq 0$. Using the formula (\ref{Zfdefeqn}) for $Z_f$, it follows that we have
\begin{align}
\label{multiaddeqn1}
M((Z_f)_f) &= \sum_{\substack{f\\ \mbox{some } f(i)= 0}} c_f\left(\prod_{f(i)= 0}X_i\right) \nonumber \\
& \times \left(\prod_{f(i) \neq 0}\sum_{\substack{0 < \ell < p^n\\ \ell \equiv f(i)/p^{v_p(f(i))}
\pmod{p^{n-v_p(f(i))}}}}\lambda_i^{(\ell - f(i)/p^{v_p(f(i))})/p^{n-v_p(f(i))}}X_{i,\ell}^{p^{v_p(f(i))}}  \right)^{p^d} \nonumber \\
+ R((X_{i,j})_{i,j})
\end{align}
for some polynomial $R$ in the $X_{i,j}$. We may similarly write $Q((Z_f)_f)$ as a $p$-polynomial such that any $X_i$ only ever appears with exponent $1$. Further, all terms of $Q((Z_f)_f)$ involving any of the $X_i$ are of strictly larger degree than the first sum above because each term of $Q$ involving some $Z_f$ with some $f(i) = 0$ has higher degree than any of the terms of $M$ by our choice of $M$. It follows that if $H((Z_f)_f) = 0$, then $M((Z_f)_f) = 0$. But any of the monomials involving any of the $X_i$ appearing in the expression (\ref{multiaddeqn1}) above for $M((Z_f)_f)$ uniquely determines the $f$ (in the sum over $f$) from which it arises: The $X_i$ which appear determine those $i$ for which $f(i) = 0$, and for each other $i$, the exponent $p^{v_p(f(i))}$ of $X_{i,\ell}$ determines $v_p(f(i))$, and then $\ell$ determines $f(i)$ because $f(i) \equiv \ell p^{v_p(f(i))}\pmod{p^n}$, so $f(i)$ is determined modulo $p^n$, hence determined. It follows that there cannot be any cancellation among different terms in the sum, so $M$ must be identically $0$. Therefore $H$ cannot involve any terms $Z_f$ with $f(i) = 0$ for some $i$.

Thus $H$ involves only $Z_f$ for $f$ which never vanish. But in that case, we may write $H = M + Q$ again, this time with $M$ the sum of terms of minimal degree. Then writing out $M((Z_f)_f)$, we conclude as above that each monomial uniquely determines the $f$ from which it arises, so that there cannot be cancellation between terms. It follows that one must have $M = 0$, so we conclude that $H = 0$, as claimed.

Now that we have the claim, it is easy to show that $b$ generates $\mathscr{V}_{n,\vec{\lambda}}$. Indeed, suppose that it did not, and let $G \subset \mathscr{V}_{n,\vec{\lambda}}$ be the subgroup it does generate. Because $\mathscr{V}_{n,\vec{\lambda}}$ is smooth and connected, we must have ${\rm{dim}}(G) < {\rm{dim}}(\mathscr{V}_{n,\vec{\lambda}})$. In particular, for the \'etale morphism $\phi\colon \mathscr{V}_{n,\vec{\lambda}} \rightarrow W := \Ga^{p^{nr}-p^{(n-1)r}}$, $(Z, (Z_f)_f) \mapsto (Z_f)_f$, the restriction $\phi|G$ is not surjective. It follows that $W/\phi(G)$ admits a $\Ga$ quotient. That is, there is a nonzero $K$-homomorphism $W \rightarrow \Ga$ killing $\phi(G)$ and in particular the image under $\phi$ of $b$. That is, there is a nonzero $p$-polynomial $H$ in the $Z_f$ killing the image of $b$. But that is exactly what the claim above says cannot happen. Thus $G = \mathscr{V}_{n,\vec{\lambda}}$, and the proof of the proposition is complete.
\end{proof}

\section{Applications to permawound groups}
\label{appstoasunivsec}

In this section we apply the rigidity theorem \ref{rigidity} to the study of permawound groups, which were introduced in \cite{rosasuniv}. (The definition and some of the fundamental properties of permawound groups were recalled in \S\ref{asunivsection}.) In particular, we will prove that permawound groups are unirational, and that wound permawound groups are commutative.

We begin by proving that permawound groups admit few homomorphisms into other groups. We will only require a very special case of this, but we prove it in general as the result is of interest in its own right.

\begin{proposition}
\label{homsfinite}
Let $K$ be a field. For any wound unipotent $K$-groups $U$ and $V$ with $U$ commutative and permawound, $\Hom(U, V)$ is finite.
\end{proposition}

\begin{proof}
By \cite[Props.\,6.7,6.9]{rosasuniv}, we may assume that $K$ is separably closed. If $K$ is perfect or of infinite degree of imperfection, then $U = 0$ \cite[Cor.\,B.2.7]{cgp}, \cite[Prop.\,6.3]{rosasuniv}, so there is nothing to prove. Thus we may assume that $K$ is imperfect of finite degree of imperfection $r$. One checks that, given an exact sequence
\[
1 \longrightarrow V' \longrightarrow V \longrightarrow V'' \longrightarrow 1
\]
with $V' \subset V$ central, if $\Hom(U, V')$ and $\Hom(U, V'')$ are finite, then so is $\Hom(U, V)$. By \cite[Prop.\,B.3.2]{cgp} and induction, therefore, we may assume that $V$ is commutative. If we let $V' \subset V$ denote the $K$-subgroup generated by all $K$-homomorphisms from $U$ into $V$, then $V'$ is a quotient of some power of $U$, hence permawound, and $\Hom(U, V) = \Hom(U, V')$. Thus we may assume that $V$ is permawound. By Theorem \ref{rigidityasuniv}, we may then assume that $V = \R_{K^{1/p}/K}(\alpha_p)$ or $\mathscr{V}$. In the former case, $\Hom(U, V) = 0$ because $\R_{K^{1/p}/K}(\alpha_p)$ is totally nonsmooth. Thus we now assume that $V = \mathscr{V}$. 

One readily verifies that, given an exact sequence
\[
0 \longrightarrow U' \longrightarrow U \longrightarrow U'' \longrightarrow 0,
\]
if $\Hom(U', \mathscr{V})$ and $\Hom(U'', \mathscr{V})$ are finite, then so is $\Hom(U, \mathscr{V})$. Therefore, applying Proposition \ref{rigidityasuniv} again, it suffices to show that $\Hom(\R_{K^{1/p}/K}(\alpha_p), \mathscr{V})$ and $\End(\mathscr{V})$ are finite. These finiteness statements follow from \cite[Lem.\,7.9, Prop.\,9.7]{rosasuniv}.
\end{proof}

The key to proving the commutativity of wound permawound groups is the following ``rigidity'' result (see Lemma \ref{rigidityproper}) for permawound groups.

\begin{lemma}
\label{rigformapsasuniv}
Let $U$ be a commutative, wound, permawound unipotent $K$-group scheme. Suppose given a wound unipotent $K$-group $W$ and a $K$-scheme map $f\colon U \times U \rightarrow W$ such that $f$ restricts to the constant map to $1$ on $0 \times U$ and on $U \times 0$. Then $f$ is the constant map to $1$.
\end{lemma}

\begin{proof}
By \cite[Props.\,6.7,6.9]{rosasuniv}, we may assume that $K$ is separably closed. If $K$ is perfect or of infinite degree of imperfection, then $U = 0$ \cite[Cor.\,B.2.7]{cgp}, \cite[Prop.\,6.3]{rosasuniv}, so there is nothing to prove. Thus we may assume that $K$ is imperfect of finite degree of imperfection $r$. Fix a choice of $p$-basis $\vec{\lambda} := [\lambda_1, \dots, \lambda_r]$ for $K$. We first prove the lemma when $U = \mathscr{V}_{n,\vec{\lambda}}$.

By Proposition \ref{multiadd}, there is a multi-additive map
\[
b\colon X := \prod_{i=1}^r(\mathscr{V}_{n,\lambda_i})_K \rightarrow \mathscr{V}_{n,\vec{\lambda}}
\]
that generates $\mathscr{V}_{n,\vec{\lambda}}$. Note in particular that $b$ vanishes when any of the coordinates is set to $0$. Consider the map $F\colon (\mathscr{V}_{n,\vec{\lambda}})^2 \times X^2 \rightarrow W$ defined by the formula
\[
F(y_1, y_2, x_1, x_2) := f(y_1+b(x_1), y_2+b(x_2))\cdot f(y_1, y_2+b(x_2))^{-1}\cdot f(y_1, y_2)\cdot f(y_1+b(x_1), y_2)^{-1}.
\]
Then $F$ vanishes whenever any of the coordinates on $x_1$ is set to $0$, and similarly for $x_2$. Because each $\mathscr{V}_{n,\lambda_i}$ is unirational by Proposition \ref{V_nunirdegimp1}, it follows from Theorem \ref{rigidityfindegofimpbody} that, for any $y_1, y_2 \in \mathscr{V}_{n,\vec{\lambda}}(K)$, the map $F_{y_1, y_2}\colon X^2 \rightarrow W$ obtained by restricting $F$ to the fiber above $(y_1, y_2)$ vanishes. This says that, for any $x_1 \in X$, the map $G_{x_1}\colon (\mathscr{V}_{n,\vec{\lambda}})^2 \rightarrow W$ defined by
\[
(y_1, y_2) \mapsto f(y_1+b(x_1), y_2)\cdot f(y_1, y_2)^{-1}
\]
is invariant under translation by $b(X)$ on $y_2$. Because $b(X)$ generates $\mathscr{V}_{n,\vec{\lambda}}$, it follows that $G_{x_1}$ is independent of $y_2$. Setting $y_2 = 0$, therefore, it is the constant map to $1 \in W$. Because $x_1 \in X$ was arbitrary, this says that $f$ is invariant under translation by $b(X)$ in the first coordinate. This implies that $f$ is independent of the first coordinate. Since $f|0 \times \mathscr{V}_{n,\vec{\lambda}} = 1$, it follows that $f \equiv 1$. This completes the proof of the proposition when $U = \mathscr{V}_{n,\vec{\lambda}}$.

We now prove the proposition in general by dimension induction on $U$. If $U = 0$ then there is nothing to do. Otherwise, there is an exact sequence
\[
0 \longrightarrow U' \longrightarrow U \longrightarrow U'' \longrightarrow 0
\]
with $U' \subset U$ nontrivial, central, $p$-torsion, and smooth and connected, and $U''$ wound \cite[Prop.\,B.3.2]{cgp}. Then $U'$ is permawound \cite[Props.\,5.5,6.9]{rosasuniv}, so by Proposition \ref{quotofVnpow}, there is a nonzero $K$-homomorphism $\psi\colon \mathscr{V}_{n,\vec{\lambda}} \rightarrow U'$. Let $V := \psi(\mathscr{V}_{n,\vec{\lambda}}) \subset U$. Then $\overline{U} := U/V$ is permawound, and it is wound because $U$ is wound and $V$ -- as a quotient of an permawound group -- is permawound.

We first note that any morphism of pointed $K$-schemes $g\colon (V, 0) \rightarrow (W, 1)$ is a homomorphism. Indeed, $V$ is a quotient of $\mathscr{V}_{n,\vec{\lambda}}$, so it suffices to verify this with $V$ replaced by this group. Let $H\colon (\mathscr{V}_{n,\vec{\lambda}})^2 \rightarrow W$ denote the map $(v_1, v_2) \mapsto g(v_1+v_2)g(v_2)^{-1}g(v_1)^{-1}$. Then $H$ vanishes when either coordinate is restricted to $0$, hence is trivial by the already-treated case of the lemma. Thus $g$ is a homomorphism.

Now suppose given a morphism $f\colon U^2 \rightarrow W$ as in the proposition. Consider the map $F\colon U^2 \times V \rightarrow W$ defined by the formula
\[
(u_1, u_2, v) \mapsto f(u_1+v, u_2)\cdot f(u_1, u_2)^{-1}.
\]
For any $u_1, u_2 \in U(K)$, let $F_{u_1, u_2}\colon V \rightarrow W$ denote the map obtained by restricting $F$ to the fiber above $(u_1, u_2)$. Then $F_{u_1, u_2}(0) = 1$, so $F_{u_1, u_2}$ is a homomorphism. Because $U^2(K)$ is Zariski dense in $U^2$, Proposition \ref{homsfinite} implies that there is a Zariski dense set of $(u_1, u_2)$ such that $F_{u_1, u_2}$ equals a fixed homomorphism $\phi\colon V \rightarrow W$. It follows that $F_{u_1, u_2} = \phi$ for all $u_1, u_2$. Setting $u_2 = 0$, we find that $F_{u_1, u_2} = 1$ for all $u_1, u_2 \in U(K)$. Therefore $F \equiv 1$. That is, $f$ is invariant under $V$-translation in the first coordinate. By a similar argument, it is also $V$-invariant in the second coordinate. It follows that $f$ descends to a map $\overline{U}^2 \rightarrow W$, and this latter map vanishes by induction. This completes the proof of the lemma.
\end{proof}

\begin{corollary}
\label{mapfromasuniv=hom}
Let $U$ and $W$ be wound unipotent $K$-group schemes with $U$ commutative and permawound. Then any morphism $f\colon (U, 0) \rightarrow (W, 1)$ of pointed $K$-schemes is a homomorphism.
\end{corollary}

\begin{proof}
Apply Lemma \ref{rigformapsasuniv} to the map $U^2 \rightarrow W$, $(u_1, u_2) \mapsto f(u_1+u_2)f(u_2)^{-1}f(u_1)^{-1}$.
\end{proof}

We may now prove the commutativity of wound permawound $K$-groups.

\begin{corollary}
\label{asunivcomm}
Every wound permawound unipotent $K$-group scheme is commutative.
\end{corollary}

\begin{proof}
We proceed by dimension induction. If $U = 1$ then the assertion is trivial, so assume that $U \neq 1$. By \cite[Prop.\,B.3.2]{cgp}, there is an exact sequence of smooth connected unipotent $K$-group schemes
\[
1 \longrightarrow U' \longrightarrow U \longrightarrow U'' \longrightarrow 1
\]
with $U' \subset U$ nontrivial, central, and $U''$ wound. Then $U'$ is also wound, and it is permawound \cite[Props.\,5.5, 6.9]{rosasuniv}, as is $U''$. By induction, $U''$ is commutative. Because $U'$ is central and $U''$ is commutative, the commutator map $U \times U \rightarrow U$ descends to a map $f\colon U'' \times U'' \rightarrow U'$. This map vanishes on both $0 \times U''$ and $U'' \times 0$, so by Lemma \ref{rigformapsasuniv} it is trivial. That is, $U$ is commutative.
\end{proof}

\begin{corollary}
\label{asunivsepext}
Let $L/K$ be a separable field extension of the same degree of imperfection as $K$, and let $U$ be a smooth unipotent $K$-group scheme. Then $U$ is permawound over $K$ if and only if it is so over $L$.
\end{corollary}

\begin{proof}
If $K$ is perfect then so is $L$, so the assertion follows from \cite[Prop.\,5.2]{rosasuniv}. Thus we may assume that $K$ (hence also $L$) is imperfect. Since permawound groups over imperfect fields are connected \cite[Prop.\,6.2]{rosasuniv}, we may further assume that $U$ is connected. Let $U_{\rm{split}} \subset U$ denote the maximal split unipotent $K$-subgroup. Then $U_w := U/U_{\rm{split}}$ is wound, and $U$ is permawound (over $K$ or $L$) if and only if $U_w$ is. Thus we may assume that $U$ is wound over $K$, hence also over $L$. Then if $U$ is permawound over either $K$ or $L$, then it is commutative by Corollary \ref{asunivcomm}. The corollary therefore follows from \cite[Props.\,6.7, 6.9]{rosasuniv}.
\end{proof}

We are now ready to complete the proof of Theorem \ref{asunivunirrcomm}.

\begin{theorem}$($Theorem $\ref{asunivunirrcomm}$$)$
\label{asunivunirrcommbody}
Let $K$ be an imperfect field.
\begin{itemize}
\item[(i)] Every permawound unipotent $K$-group is unirational.
\item[(ii)] Every wound permawound unipotent $K$-group is commutative.
\end{itemize}
\end{theorem}

\begin{proof}
Assertion (ii) is Corollary \ref{asunivcomm}, so it only remains to prove (i). By Corollary \ref{asunivsepext} and Theorem \ref{unirationalitydescendsbody}, we may assume that $K$ is separably closed. Let $U$ be an permawound unipotent $K$-group. Because $K$ is imperfect, $U$ is connected \cite[Prop.\,6.2]{rosasuniv}. We prove that $U$ is unirational by dimension induction. The $0$-dimensional case is trivial, so assume that $U \neq 1$. First assume that $U$ is not wound, and let $1 \neq U_{\rm{split}} \trianglelefteq U$ denote the maximal split unipotent $K$-subgroup scheme. As usual, the left-multiplication action makes $U$ into a $U_{\rm{split}}$-torsor over $U_w$. Because $U_w$ is affine and $U_{\rm{split}}$ is split, this torsor is trivial, so $U = U_{\rm{split}} \times U_w$ as $K$-schemes. Because $U_w$ is also permawound, we are done by induction in this case.

Now assume that $U$ is wound and nontrivial. By \cite[Prop.\,B.3.2]{cgp}, there is an exact sequence of smooth connected unipotent $K$-groups
\[
1 \longrightarrow U' \longrightarrow U \longrightarrow U'' \longrightarrow 1
\]
with $U' \subset U$ nontrivial, central, and $p$-torsion, and $U''$ wound. Then $U'$ is wound and permawound \cite[Props.\,5.5, 6.9]{rosasuniv}. By Theorem \ref{rigidityasuniv}, $U'$ contains a $K$-subgroup scheme $V$ that is isomorphic to either $\R_{K^{1/p}/K}(\alpha_p)$ or $\mathscr{V}$. Let $\overline{U} := U/V$. Then $\overline{U}$ is unirational by induction. Let $G \trianglelefteq U$ denote the maximal unirational $K$-subgroup scheme. By Proposition \ref{torsorsdieratlcurve}, $\overline{U}$ is generated by the images of maps from open subschemes of $\P^1_K$ into $U$. It follows that $G \twoheadrightarrow \overline{U}$. It therefore only remains to prove that $G$ contains $V$, for which it suffices to prove that $U'$ is unirational. In fact, $\mathscr{V}_{n,\vec{\lambda}}$ is unirational for all $n > 0$ by Propositions \ref{multiadd} and \ref{V_nunirdegimp1}, and $U$ is a quotient of some $(\mathscr{V}_{n,\vec{\lambda}})^m$ by Proposition \ref{quotofVnpow}.
\end{proof}

\noindent \address
\vspace{.3 in}

\noindent \email

\end{document}